\documentclass[10.9pt,a4paper]{amsart}
\usepackage[english]{babel}
\usepackage[normalem]{ulem}

\usepackage[square,sort,comma,numbers]{natbib}

\usepackage[all]{xy}
\usepackage{pstricks}
\usepackage{enumerate}
\usepackage{amsfonts,amssymb,amsmath,pinlabel,array,hhline}
\usepackage{slashed}
\usepackage{tabulary}
\usepackage{fancyhdr}
\usepackage{color}
\usepackage{a4wide}
\usepackage{bbm}
\usepackage[position=b]{subcaption}
\usepackage[colorlinks,
    linkcolor={blue!50!black},
    citecolor={blue!50!black},
    urlcolor={blue!80!black}]{hyperref}

    \usepackage{tikz}
\usetikzlibrary{
  cd,
  calc,
  positioning,
  fit,
  arrows,
  decorations.pathreplacing,
  decorations.markings,
  shapes.geometric,
  backgrounds,
  bending
}
\usepackage{tikzsymbols}

\usepackage{overpic}
\usepackage{todonotes}

\newtheorem{theorem}{Theorem}[section]

\newtheorem{question}{Question}
\newtheorem{question*}{Question}
\newtheorem*{lemma*}{Lemma}
\newtheorem{corollary}[theorem]{Corollary}
\newtheorem{lemma}[theorem]{Lemma}
\newtheorem{proposition}[theorem]{Proposition}

\newtheorem*{claim*}{Claim}
\theoremstyle{definition}
\newtheorem{definition}[theorem]{Definition}

\newtheorem*{notation*}{Notation}

\newtheorem{example}[theorem]{Example}

\newtheorem{remark}[theorem]{Remark}
\newtheorem*{remark*}{Remark}

\definecolor{bettergreen}{rgb}{0.0, 0.5, 0.0}

\newcommand{\N}{\mathbb{N}}
\newcommand{\Z}{\mathbb{Z}}
\newcommand{\Q}{\mathbb{Q}}

\newcommand{\D}{\mathbb{D}}
\newcommand{\F}{\mathbb{F}}

\newcommand{\bsm}{\left(\begin{smallmatrix}}
\newcommand{\esm}{\end{smallmatrix}\right)}

\newcommand{\Bl}{\operatorname{Bl}}

\newcommand{\lk}{\ell k}

\newcommand{\T}{\operatorname{T}}

\newcommand{\alg}{\mathit{alg}}

\DeclareSymbolFont{EulerScript}{U}{eus}{m}{n}
\DeclareSymbolFontAlphabet\mathscr{EulerScript}

\title{Algebraic concordance groups for links}

\author{Ga\"etan Simian}
\address{Ga\"etan Simian -- Section de math\'ematiques, Universit\'e de Gen\`eve, Suisse}
\email{gaetan.simian@unige.ch}

\date{}

\begin{document}

\begin{abstract}
We use families of generalized Seifert matrices to define algebraic concordance groups for links with non trivial Alexander polynomial. We then use these groups to recover the invariance by concordance of the signature, the Fox-Milnor condition for concordant links and the invariance by concordance of the Witt class of the Blanchfield pairing for links. Finally, we use these invariants to investigate the isomorphism types of these groups.
\end{abstract}

\maketitle


\section{Introduction}
\subsection{Historical background and motivation}

Two knots in $S^3$ are called \emph{topologically concordant} if they cobound a locally flat annulus in $S^3 \times [0,1]$. Concordance is an equivalence relation, and the set of concordance classes of knots forms a group, called the \emph{knot concordance group}, for the operation induced by connected sum. The neutral element is given by the class of \emph{slice} knots, i.e. knots that bound a (locally flat) embedded disc in $B^4$. This group was introduced in 1966 by Fox and Milnor \cite{FoxMilnor}. One way to study the knot concordance group is via its algebraic version, which is defined as follows. A Seifert matrix $A$ for a knot is called \emph{metabolic}, or equivalently \emph{algebraically slice} if it is congruent over $\Z$ to a block matrix of the form $\bsm 0 & C\\D & E\esm$, with $C$, $D$ and $E$ square matrices half the size of $A$. Two Seifert matrices $A$ and $B$ are called \emph{Witt equivalent}, or equivalently \emph{algebraically concordant} if the direct sum $A\oplus-B$ is metabolic. Witt equivalence is an equivalence relation on the set of Seifert matrices, and the set of Witt equivalence classes form a group, called the \emph{algebraic concordance group}, for the operation induced by the direct sum. The neutral element is given by the class of metabolic Seifert matrices.

In \cite{Lev69}, Levine showed that slice knots have algebraically slice Seifert matrices. This property ensures that associating to the concordance class of a knot the algebraic concordance class of its Seifert matrices yields a well defined homomorphism from the knot concordance group onto the algebraic concordance group. Therefore, any invariant of the algebraic concordance group defines an invariant of knot concordance. These are called \emph{algebraic concordance invariants of knots}, and knots with algebraically concordant Seifert matrices are called algebraically concordant. In the early days of knot concordance, all concordance invariants were algebraic concordance invariant. Among these, let us mention the Levine-Tristram signature \cite{Lev69, Tri69}, studied by Murasugi \cite{Mur69}, the Arf invariant \cite{Rob65}, the Fox-Milnor condition on the Alexander polynomial \cite{FoxMilnor}, the Blanchfield pairing \cite{Bla57}, and invariants coming from Witt groups of finite fields. By combining these invariants, Levine showed that the algebraic concordance group is isomorphic to $\Z^\infty \oplus \Z_2^\infty \oplus \Z_4^\infty$ \cite{LevInvKnotCob}.

Let us mention that, in higher dimensions, i.e. knotted spheres $S^n$ in $S^{n+2}$ for $n >1$, this completely classifies concordance of knots. Indeed, Levine \cite{LevInvKnotCob} showed that in this case, the homomorphism associating to the concordance class of a knot the algebraically concordance class of its Seifert matrix is an isomorphism. However, this does not hold $n=1$. Indeed, the existence of elements in the kernel, i.e. algebraically slice knots which are not slice was established in 1976 by Casson and Gordon \cite{CassonGordonSlice, CassonGordonCobordism}.

Concordance of links is defined in a similar manner : two links in $S^3$ are \emph{topologically concordant} if they cobound disjoint locally flat annuli properly embedded in $S^3 \times [0,1]$, each intersecting $S^3 \times \{0\}$ and $S^3 \times \{1\}$. A link is called \emph{slice} if it bounds disjoint locally flat embedded discs in $B^4$. Even though concordance is an equivalence relation on the set of links, there is no natural group structure on the set of equivalence classes. This is because for more than one component, the connected sum operation depends on the choice of a component, and therefore does not yield a well defined operation. However, analogues of Seifert matrices do exist for links. These are called \emph{generalized Seifert matrices}, introduced by \cite{Coo82, Cim04}, and can be used to compute several link invariants, such as the multivariable Levine-Tristram signature \cite{C-F08}, the multivariable Alexander-Conway polynomial \cite{Cim04}, or the multivariable Blanchfield pairing \cite{CFT18, ConwayBlanchfield}.

Event though the analogue of Seifert matrices for links is well defined, there is so far no construction of an algebraic concordance group for links. The present work is an attempt to fill this gap. Using generalized Seifert matrices of links with non trivial Alexander polynomial, we construct algebraic concordance groups for links, and use them successfully to recover known results such as the invariance by concordance of the signature \cite{C-F08, CNT}, the Fox-Milnor condition \cite{Kawauchi} and the invariance by concordance of the Witt class of the Blanchfield form, see e.g. \cite{hillman2012}.

\subsection{Paper outline and statement of the results}

Throughout this note, we work in the setting of colored links. A \emph{$\mu$-colored link} is an oriented link together with a surjective map assigning to each component a color in $\{1, \ldots, \mu\}$. This framework is a natural setting which allows to consider all at once oriented links, corresponding to the case $\mu =1$, and ordered links, which correspond to the case where $\mu$ equals the number of components. All the necessary background on colored links and generalized Seifert matrices necessary for this note is contained in Section \ref{sec:Background}.

Section \ref{sec:AlgConcLinks} is our main section, devoted to the construction of algebraic concordance groups for $\mu$-colored links. Our idea is to mimic the construction of the classical algebraic concordance group to links, using generalized Seifert matrices. There are three main differences compared to knots. First, generalized Seifert matrices for $\mu$-colored links come in families of $2^\mu$ matrices, depending on a \emph{sign} $\varepsilon \in \{\pm\}^\mu$. Even though in the case of oriented links one of the two Seifert matrices determines the whole pair (via transposition), this is not the case for $\mu \geq 2$. Hence, we consider families of matrices to define algebraic concordance. More precisely, we consider the set $\mathcal{M}_\mu$ of all families of integral matrices $\{A^\varepsilon\}_{\varepsilon \in \{\pm\}^\mu}$ of the same size, satisfying $A^{-\varepsilon} = (A^\varepsilon)^{\T}$ for all $\varepsilon$.

The second difference is that if $A$ is a Seifert matrix for a knot, the matrix $A-A^{\T}$ is non degenerate (in fact non singular). This property is key in proving that Witt equivalence is an equivalence relation. Analogues of this property do not always hold for families of generalized Seifert matrices, as illustrated in Example \ref{example:LinkDegenerateFamily}. To circumvent this issue, we restrict to families of matrices with \emph{non trivial Alexander polynomial}. In the following definition, $\Lambda := \Z[t_1^{\pm1}, \ldots, t_\mu^{\pm1}]$ denotes the ring of $\mu$-variable Laurent polynomials.
\begin{definition}
\label{def:MAP}
For a family of integral matrices $A = \{A^\varepsilon\}_{\varepsilon \in \{\pm\}^\mu}$ in $\mathcal{M}_\mu$ its \emph{Alexander polynomial} is the $\mu$-variable Laurent polynomial $\Delta_A(t_1, \ldots, t_\mu) \in \Lambda$, defined as the determinant of the matrix
$$A(t_1, \ldots, t_\mu) := \sum_{\varepsilon \in \{\pm\}^\mu}\varepsilon_1 \ldots \varepsilon_\mu t_1^{\frac{1-\varepsilon_1}{2}} \ldots t_\mu^{\frac{1-\varepsilon_\mu}{2}}A^\varepsilon.$$
We say that the family $A$ has \emph{non trivial Alexander polynomial} if $\Delta_A(t_1, \ldots, t_\mu)$ is not the zero polynomial.
\end{definition}

We write $\mathcal{M}_\mu^\ast \subset \mathcal{M}_\mu$ for the subset of families of matrices having non trivial Alexander polynomial. To motivate this terminology, observe that \cite{Cim04} (see also \cite[Corollary 3.6]{C-F08}) implies that if $A$ is a family of generalized Seifert matrices for a link $L$, then $\Delta_A(t)$ is, up to multiplication by powers of $(t_i-1)$, the Alexander polynomial of $L$.

\begin{remark}
\label{rk:IntroAP}
In the case $\mu = 1$, the set $\mathcal{M}_1^{\ast}$ contains all pairs of Seifert matrices coming from knots. Indeed, any knot has non trivial Alexander polynomial.
\end{remark}

On the set $\mathcal{M}_\mu^\ast$, we define an equivalence relation, called \emph{algebraic concordance}, starting with the definition of Witt equivalent families of matrices. We say that two families of matrices $A = \{A^\varepsilon\}_{\varepsilon \in \{\pm\}^\mu}$ and $B = \{B^\varepsilon\}_{\varepsilon \in \{\pm\}^\mu}$ in $\mathcal{M}_\mu$ are \emph{congruent} if there exists a unimodular integral matrix $C$ such that for all $\varepsilon \in \{\pm\}^\mu$, $CA^\varepsilon C^{\T} = B^\varepsilon$. We emphasize the fact that $C$ does not depend on $\varepsilon$. 

\begin{definition}
\label{def:Metab}
A family $A \in \mathcal{M}_\mu$ of integral square matrices is \emph{metabolic} if it is congruent to a family of matrices of the shape
$$\begin{pmatrix} 0&E^\varepsilon\\F^\varepsilon&G^\varepsilon\end{pmatrix},$$
with $E^\varepsilon$, $F^\varepsilon$ and $G^\varepsilon$ square integer matrices. In particular, the matrices of $A$ have to be of even size, and $E^\varepsilon$, $F^\varepsilon$ and $G^\varepsilon$ half that size.
\end{definition}

\begin{remark}
\label{rk:MetabolicKnot}
For $\mu = 1$, the set $\mathcal{M}_1$ consists of pairs of matrices $\{A, A^{\T}\}$. Moreover, two families $\{A, A^{\T}\}$ and $\{B, B^{\T}\}$ are congruent if and only if the matrices $A$ and $B$ are congruent as $CAC^{\T} = B$ implies $CA^{\T}C^{\T} = B^{\T}$. Therefore, a pair of matrices $\{A, A^{\T}\}$ is metabolic if and only if the matrix $A$ is congruent to one that has a half rank block of zeros in the top left corner. This is, up to some non degeneracy hypothesis, the classical definition of metabolicity, see for example \cite[Definition 2.5]{LivingstonKnotConcordance}.
\end{remark}

To define Witt equivalence of families of matrices, we introduce two notations. For a family of matrices $A = \{A_\varepsilon\}_{\varepsilon \in \{\pm\}^\mu}$ in $\mathcal{M}_\mu$, the family $-A \in \mathcal{M}_\mu$ is defined as $\{-A^\varepsilon\}_{\varepsilon \in \{\pm\}^\mu}$. For a family $B = \{B^\varepsilon\}_{\varepsilon \in \{\pm\}^\mu}$, the \emph{direct sum} of $A$ and $B$ is $A\oplus B := \{A^\varepsilon \oplus B^\varepsilon\}_{\varepsilon \in \{\pm\}^\mu}$.

\begin{definition}
\label{def:Witt}
Two families of matrices $A$ and $B$ of $\mathcal{M}_\mu^\ast$ are called \emph{Witt equivalent} if the direct sum $A\oplus -B$ is metabolic.
\end{definition}

\begin{remark}
\label{rk:WittKnots}
\begin{enumerate}[(i)]
\item As a metabolic family necessarily consists of matrices of even size, matrices of Witt equivalent families must have size of the same parity.
\item In the case $\mu = 1$, by Remark \ref{rk:MetabolicKnot}, two pairs of matrices $\{A, A^{\T}\}$ and $\{B, B^{\T}\}$ are Witt equivalent if and only if $A\oplus -B$ is a metabolic matrix, which is the classical definition of Witt equivalence of Seifert matrices, see \cite[Section 1]{Lev69} or \cite[Definition 3.4.3]{LivingstonNaik}
\end{enumerate}
\end{remark}

The third and last difference is that in the case of knots, any two Seifert surfaces for the same knot give Witt equivalent Seifert matrices. Equivalently, this means that S-equivalence implies Witt equivalence. This fails for $\mu \geq 2$, as illustrated in Example \ref{example:Hopf}. Therefore, to get a notion of algebraic concordance which only depends on links rather than C-complexes, we are forced to enlarge the equivalence relation by including generalized S-equivalence. Postponing the technical definition of generalized S-equivalence to Section \ref{sec:SEquivalence} (see Definition \ref{def:GSEquivalence}), we define algebraic concordance of families of matrices.

\begin{definition}
Two families of matrices $A$ and $B$ in $\mathcal{M}_\mu^\ast$ are \emph{algebraically concordant} if there exists a sequence of families of matrices $A = A_0, A_1, \ldots, A_n = B$, each having non trivial Alexander polynomial, and such that for all $0 \leq i < n$, $A_i$ and $A_{i+1}$ are either generalized S-equivalent, or Witt equivalent.
\end{definition}

\begin{remark}
\label{rk:AlgConcKnots1}
In the case $\mu = 1$, we recover the classical notion of algebraic concordance for Seifert matrices. Indeed, two pairs of Seifert matrices of knots $\{A, A^{\T}\}$ and $\{B, B^{\T}\}$ lie in $\mathcal{M}_1^{\ast}$ by Remark~\ref{rk:IntroAP}. By definition these are algebraically concordant as pairs if and only if they are related by a sequence of Witt equivalence and generalized S-equivalence. As $\mu = 1$, generalized S-equivalence implies Witt equivalence, and it follows that $\{A, A^{\T}\}$ and $\{B, B^{\T}\}$ are algebraically concordant pairs if and only they are Witt equivalent pairs. By Remark \ref{rk:WittKnots}, they are Witt equivalent pairs if and only if $A$ and $B$ are Witt equivalent matrices, which is the classical definition of algebraically concordant Seifert matrices.
\end{remark}

Using this notion, we state our main result.

\begin{theorem}
\label{thm:IntroAlgConcLinks}
For any positive integer $\mu$, algebraic concordance is an equivalence relation on the set of families of matrices $\mathcal{M}_\mu^\ast$. Moreover, direct sum yields a well defined abelian group structure on the quotient.
\end{theorem}

Using Theorem \ref{thm:IntroAlgConcLinks}, we formulate our main definition.

\begin{definition}
\label{def:AlgConcGroup}
We call the group of Theorem \ref{thm:IntroAlgConcLinks} the \emph{$\mu$-colored algebraic concordance group}, denoted by $\mathcal{C}_\mu^{\alg}.$
\end{definition}

\begin{remark}
\label{rk:AlgConcGroupMu=1}
We emphasize the fact the the group $\mathcal{C}_1^\alg$ does not coincide with the classical algebraic concordance group of knots. Indeed, the group $\mathcal{C}_1^\alg$ contains the classes of Seifert matrices of all oriented links with arbitrary number of components and non trivial Alexander polynomial. In particular, it contains classes of pairs of matrices of odd size. On the other hand, the classical algebraic concordance group only contains the classes of Seifert matrices of knots, an thus only classes of pairs of matrices of even size. 
As the parity of the size of a pair of matrices is preserved by S-equivalence and Witt equivalence, it is an invariant of the algebraic concordance class of a pair of matrices. Thus, the two groups do not coincide.
\end{remark}

Of course, the original motivation for the introduction of the classical algebraic concordance group is the fact that concordant knots have algebraically concordant Seifert matrices. Our construction satisfies the analogue for colored links, at least in the smooth category. This is the content of the following theorem, which is proved in Section \ref{sub:AlgebraicConcordanceOfLinks}.

\begin{theorem}
\label{prop:IntroAlgConcLinks}
Smoothly concordant colored links with non trivial Alexander polynomial have algebraically concordant families of generalized Seifert matrices.
\end{theorem}

Note that the notion of \emph{smooth} concordance might be unsettling for the reader used to Seifert matrices and abelian invariants. Indeed, Seifert matrices are classically considered in the topological category, yielding topological concordance invariants. However, our actual proof heavily relies on the smooth setting and all our attempts to prove an analogue of Theorem~\ref{prop:IntroAlgConcLinks} in the topological category have been unsuccessful. In particular, we were not able to adapt the classical proof from \cite[Chapter 8]{Lickorish}, details are given in \cite{GSThesis}.

In Section \ref{sec:AlgConcInvariants}, we turn toward the extension of several known concordance invariants of colored links to the groups $\mathcal{C}_\mu^\alg$. More precisely, in Section \ref{sub:Signature}, we show that the multivariable signature from \cite{C-F08} gives a invariant of the $\mu$-colored algebraic concordance group. Using this invariant, we obtain a three dimensional smooth proof of the following result, classically obtained in the topological category via four dimensional considerations \cite{C-F08, CNT}. 

\begin{corollary}
\label{cor:Signature}
If $L$ and $L'$ are smoothly concordant colored links, then for any $\omega \in (S^1 \setminus \{1\})^\mu$ which is not a root of the Alexander polynomials $\Delta_L(t)$ and $\Delta_{L'}(t)$, the signatures $\sigma_L(\omega)$ and $\sigma_{L'}(\omega)$ are equal.
\end{corollary}

In the following Section \ref{sub:FoxMilnor}, we establish a Fox-Milnor condition for the Alexander polynomials of algebraically concordant families of matrices. Applying this condition to colored links, we recover a weaker version of \cite[Theorem B]{Kawauchi}. As in the case of the signature, our proof is smooth and three dimensional, while the one of Kawauchi uses homological techniques and holds in the topological category (see also \cite{NakagawaAP} for a proof using Fox calculus). To state this result, recall that $\Lambda := \Z[t_1^{\pm1}, \ldots, t_\mu^{\pm1}]$ denotes the ring of $\mu$-variable Laurent polynomials. We write $\Lambda_S$ for its localization with respect to the multiplicative set generated by the elements $(t_1-1),\ldots, (t_\mu-1)$.

\begin{corollary}
\label{cor:FoxMilnor}
For smoothly concordant colored links $L$ and $L'$, there exists Laurent polynomials $f, g \in \Lambda$, satisfying $\vert f(1, \ldots, 1) \vert  = \vert g(1, \ldots, 1) \vert = 1$, such that the polynomials $\Delta_L(t) f(t)f(t^{-1})$ and $\Delta_{L'}(t) g(t)g(t^{-1})$ are equal up to multiplications by units of $\Lambda_S$.
\end{corollary}

In Section \ref{sub:Blanchfield}, using results of \cite{CFT18}, we define Blanchfield pairings for families of matrices with non trivial Alexander polynomial. We show that this construction yields an algebraic concordance invariant, thereby recovering a result previously known, see e.g. \cite[Theorem 2.4]{hillman2012}. Again, our proof is mostly three dimensional and smooth, while the one of Hillman is four dimensional and holds in the topological category.

\begin{corollary}
\label{cor:Blanchfield}
Suppose that $L$ and $L'$ are smoothly concordant colored links with non trivial Alexander polynomials. Then the Blanchfield pairings $\Bl_L$ and $\Bl_L'$ over $\Lambda_S$ are Witt equivalent.
\end{corollary}

As in the classical case of links, the signature and the Fox-Milnor condition can be combined to unveil part of the structure of $\mathcal{C}_\mu^\alg$. These give the following theorem, which is a direct consequence of Lemma \ref{lemma:ZinfSummand} and Lemma \ref{lemma:Z2Summand}.

\begin{theorem}
\label{prop:StructureAlgConcGroup}
For all $\mu \geq 1$, the $\mu$-colored algebraic concordance group $\mathcal{C}_\mu$ contains a summand isomorphic to $\Z^\infty \oplus \Z_2^\infty$.
\end{theorem}

\subsection{Other versions of algebraic concordance}
We conclude this introduction by mentioning other approaches to the definition of algebraic concordance of links. In the case of knots, there exists several equivalent criterion for two knots to be algebraically concordant. Indeed, as explained in the introduction of \cite{CCS26}, algebraic sliceness of a knot $K$ can also be defined by considering equivariant intersection forms of spin fillings of the $0$-surgery along $K$, or via its Blanchfield pairing. Therefore, one could also try to define algebraic concordance of links starting from each of these two approaches. This is done in \cite{CCS26}, yielding so called \emph{homology surgery invariant} and \emph{Blanchfield invariant}.

Note that for two components, $2$-colored links with non trivial Alexander polynomials and equal linking numbers, algebraic concordance defined via generalized Seifert matrices (see Definition \ref{def:AlgConcLinks}) implies the vanishment of the homology surgery invariant and of the Blanchfield invariant. Indeed, if two such links are algebraically concordant, then their families of generalized Seifert matrices are algebraically concordant, and thus the matrices $H_F$ and $H_{F'}$ of \cite[Theorem 1.1 (5)]{CCS26} are Witt equivalent. It follows that the homology surgery invariant is represented by a metabolic matrix, and thus vanishes. The vanishment of the Blanchfield invariant follows from \cite[Corollary 1.7]{CCS26}.

\subsection*{Organization}

This paper is organized as follows. In Section \ref{sec:Background} we give all necessary background concerning colored links, including the definition of C-complexes, generalized Seifert matrices and generalized S-equivalence. In Section \ref{sec:AlgConcLinks}, we define Witt equivalence and algebraic concordance of families of matrices and prove Theorem \ref{thm:IntroAlgConcLinks}. We then connect with links, proving Theorem \ref{prop:IntroAlgConcLinks} in Section~\ref{sub:AlgebraicConcordanceOfLinks}. Finally, in Section \ref{sec:AlgConcInvariants}, we use the signature, the Alexander polynomial and the Blanchfield pairing of colored links to define algebraic concordance invariants of families of matrices. We then use these invariants to prove Theorem \ref{prop:StructureAlgConcGroup}. We ultimately collect a few open questions in Section~\ref{sec:Questions}.

\subsection*{Acknowledgements}
The author wishes to thank his advisors David Cimasoni and Anthony Conway for helpful discussions and endless support. This work was supported by the Swiss NSF grant 200021-212085.

\subsection*{Notations and conventions}
Unless specified, all matrices have integral entries. We consider colored links, and denote by $\mu$ the number of colors. The symbol $\sim$ denotes an equivalence relation, either clear from the context or with an indicative index. We write $\Lambda := \Z[t_1^{\pm1}, \ldots, t_\mu^{\pm1}]$ for the ring of $\mu$-variable Laurent polynomial and $\Lambda_S$ for its localization with respect to the multiplicative system generated by elements of the form $(t_i-1)$. The symbol $\dot=$ stands for equality up to multiplication by units of $\Lambda$, while the symbol $\ddot = $ stands for equality up to multiplication by units of $ \Lambda_S$. The shorthand $t = (t_1, \ldots, t_\mu)$ will sometimes be used when the context allows no confusion.

\section{Background on colored links and C-complexes}
\label{sec:Background}

In this section, we give all the necessary background on colored links for our work. In Section~\ref{sec:ColoredLinks} we recall the definition of colored links, C-complexes and generalized Seifert matrices. Section~\ref{sec:SEquivalence} is devoted to generalized S-equivalence.

\subsection{Colored links, C-complexes and generalized Seifert matrices} 
\label{sec:ColoredLinks}
In this section, we recall the definition of colored links, together with the two main tools that we will use in this work: the analogues of Seifert surfaces, called C-complexes, and the analogue of Seifert matrices, called generalized Seifert matrices. We follow \cite{CSArf} and use the figures therein. \medskip

We start with a definition. 

\begin{definition}
A \emph{$\mu$-colored link} is an oriented link in $S^3$, together with a surjective map assigning to each component a color in $\{1, \ldots, \mu\}$. Two colored links are \emph{isotopic} if they are isotopic as oriented links, via an isotopy preserving the colors.
\end{definition}

We denote a colored link by $L=L_1\cup\dots\cup L_\mu$, where $L_i$ stands for the sublink of $L$ of components of color $i$. Colored links are a common generalization of oriented links, which correspond to the case $\mu = 1$, and ordered links, which correspond to the case where $\mu$ coincides with the number of components.

As the purpose of this note is to define algebraic concordance of colored links, we shall recall what are the analogues of Seifert surfaces and Seifert matrices for colored links. This is what we now do. The generalized notion of Seifert surface is due to Cooper~\cite{Coo82}, who defined it for~2-component~2-colored links. Its generalization to an arbitrary number of colors and components is due to Cimasoni, see~\cite{Cim04}.

\begin{definition}
\label{def:C-cplx}
A {\em C-complex\/} for a~$\mu$-colored link~$L=L_1\cup\dots\cup L_\mu$ is a union~$F=F_1\cup\dots\cup F_\mu$ of surfaces embedded in~$S^3$ such that~$F$ is connected and satisfies the following conditions:
\begin{enumerate}
    \item for all~$i$, the surface~$F_i$ is a (possibly disconnected) Seifert surface for~$L_i$;
    \item for all~$i\neq j$, the surfaces~$F_i$ and~$F_j$ are either disjoint or intersect in a finite number of {\em clasps\/} (see Figure~\ref{fig:Clasp});
    \item for all~$i,j,k$ pairwise distinct, the intersection~$F_i\cap F_j\cap F_k$ is empty.
\end{enumerate}
\end{definition}

\begin{figure}
\centering
\begin{overpic}[width=10cm]{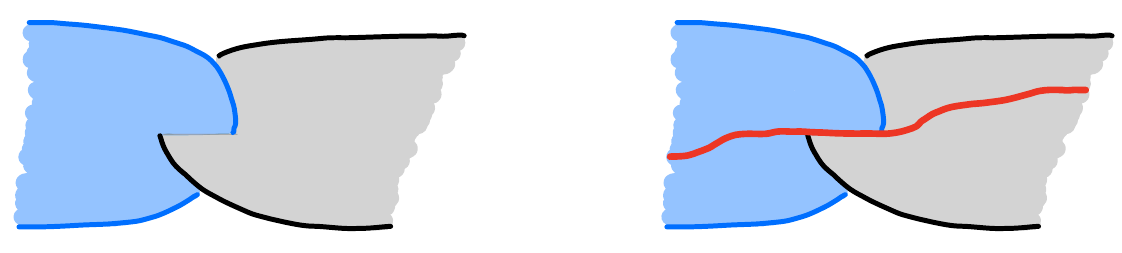}
 \put (5,15){$F_i$}
    \put (32,13){$F_j$}
    \end{overpic}
\caption{Left: a clasp intersection. Right: a well-behaved cycle crossing such a clasp intersection.}
\label{fig:Clasp}
\end{figure}

The existence of a C-complex for any given colored is a standard fact, which follows from the fact that any family of Seifert surfaces for the sublinks of different colors can be isotoped to yield only clasp intersections, see~\cite{Cim04}. On the other hand, relating two C-complexes for isotopic colored links is a more difficult task which was initiated by \cite[Lemma 3]{Cim04} and completed by \cite[Theorem 1.3]{DMO21}. We recall this result here for convenience.

\begin{lemma}[\cite{DMO21}]
\label{lemma:SEquivalenceCComplex}
Two C-complexes for isotopic colored links are related by a finite sequence of the following moves or their inverses:
\begin{itemize}
\item[(T0)] ambient isotopy;
\item[(T1)] handle attachment along one of the surfaces;
\item[(T2)] add a ribbon intersection and push along an arc (see the left of Figure \ref{T23});
\item[(T3)] pass through a clasp (see the right of Figure \ref{T23});
\item[(T4)] replace a push along an arc by a push along another arc (see Figure~\ref{T4}).
\end{itemize}
\end{lemma}

\begin{figure}[h]
\centering
\begin{overpic}[width=6.5cm]{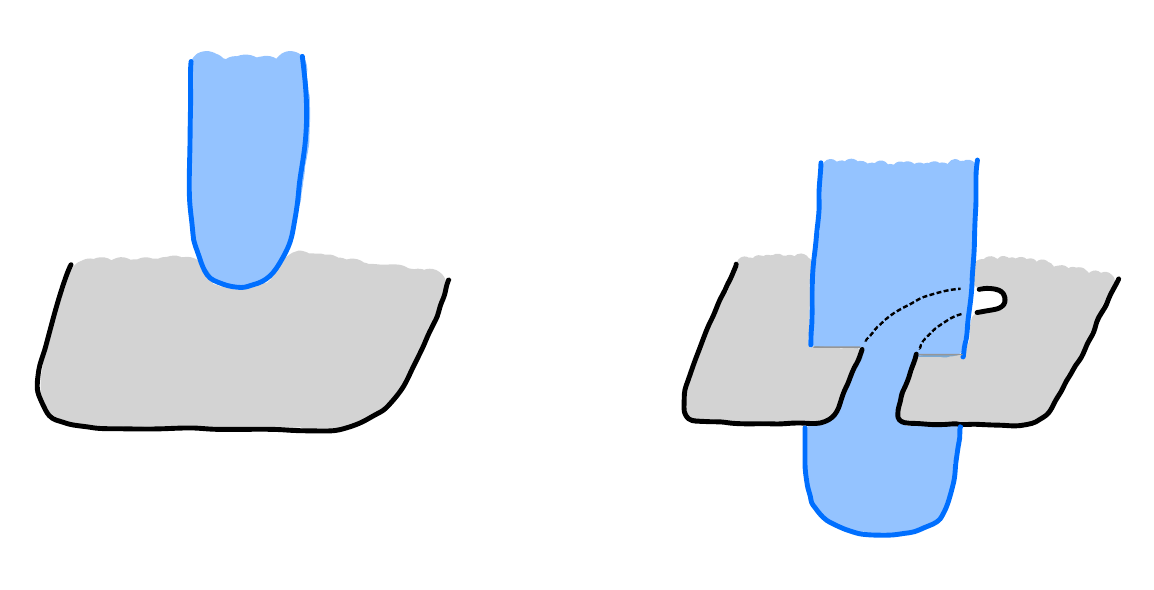}
\put(45,20){$\longrightarrow$}
\put(44,25){$\mathrm{(T2)}$}
\end{overpic}
\hskip1.5cm
\begin{overpic}[width=6.5cm]{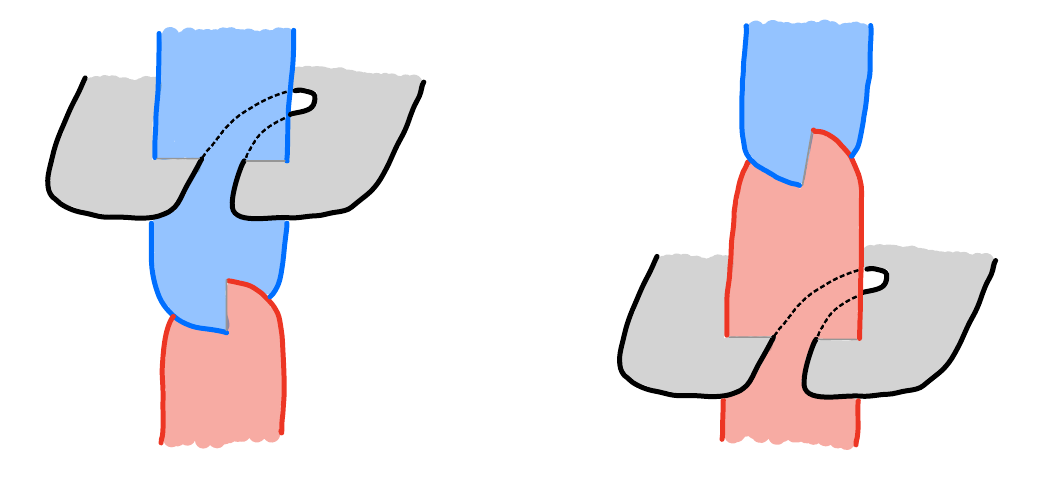}
\put(45,20){$\longrightarrow$}
\put(44,25){$\mathrm{(T3)}$}
\end{overpic}
\caption{The movements~(T2) and~(T3).}
\label{T23}
\end{figure}

\begin{figure}[h]
\centering
\begin{overpic}[width=10cm]{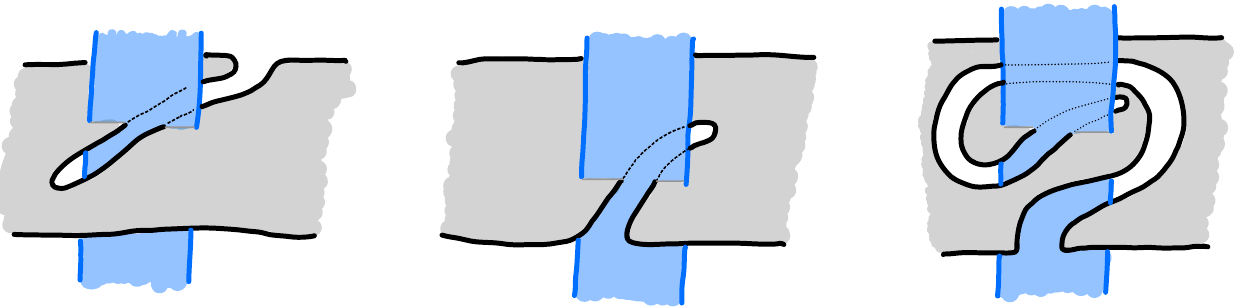}
\put(28,10){$\longleftarrow$}
\put(27,13.5){$\mathrm{(T4_a)}$}
\put(67,10){$\longrightarrow$}
\put(65,13.5){$\mathrm{(T4_b)}$}
\end{overpic}
\caption{Examples of~(T4) movements. The move on the right can be seen as the composition of two replacements of the arc in the center by an arc going to the back side of the blue strip.}
\label{T4}
\end{figure}
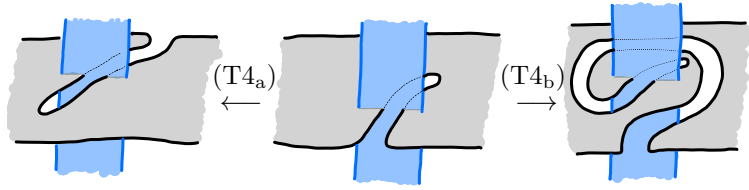

Using C-complexes one defines {\em generalized Seifert forms\/} as follows.
For any sign~$\varepsilon=(\varepsilon_1,\dots,\varepsilon_\mu)\in\{\pm\}^\mu$, let
$$ \alpha^\varepsilon\colon H_1(F)\times H_1(F)\longrightarrow\Z$$
be the bilinear form given by~$\alpha^\varepsilon(x,y)=\lk(x^\varepsilon,y)$,
where~$x^\varepsilon\subset S^3\setminus F$ denotes a cycle representing the homology class~$x\in H_1(F)$ pushed in the~$\varepsilon_i$-normal direction off~$F_i$ for all~$i=1,\dots,\mu$. For this to make sense, one needs this cycle to be well-behaved when crossing clasps, as illustrated in the right of Figure~\ref{fig:Clasp}.
Clearly, any homology class in~$H_1(F)$ can be represented by such well-behaved cycles. A \emph{generalized Seifert matrix} $A^\varepsilon = A^\varepsilon_F$ is a matrix representing the generalized Seifert form $\alpha^\varepsilon$. Note that, as in the classical case of knots, cycles can be slid simultaneously in one direction, thus $A^{-\varepsilon}=(A^\varepsilon)^{\T}$ for all~$\varepsilon\in\{\pm\}^\mu$.

Note that in the case~$\mu=1$, a C-complex for a~1-colored link~$L$ is a standard Seifert surface for the oriented link~$L$, a matrix~$A^-$ corresponding to~$\varepsilon=-$ is a classical Seifert matrix, and~$A^+$ its transpose. Thus, in this case, considering only one of the two matrices is sufficient to determine the whole family of (two) Seifert matrices.

In the general case, it is still true that a subset of the family of generalized Seifert matrices for a given link determines the whole family. Indeed, as $A^{-\varepsilon}=(A^\varepsilon)^{\T}$ for all~$\varepsilon\in\{\pm\}^\mu$, it is sufficient to know only one representative of each pair $\{A^\varepsilon, A^{-\varepsilon}\}$. Nevertheless, there is in general no formula relating the matrices $A^\varepsilon$ and $A^{\varepsilon'}$ for signs $\varepsilon'\neq \pm\varepsilon$. Hence, one has to consider a \emph{family} of generalized Seifert matrices, rather than a single one. This is what we do in the remainder of this paper.

To motivate the fact that families of generalized Seifert matrices are indeed analogues of Seifert matrices for colored links, let us mention that they permit the computations of the generalization to colored links of several invariants, such as the generalized Alexander-Conway polynomial \cite{Cim04}, the generalized Levine-Tristram signature \cite{C-F08}, and the generalized Blanchfield pairing \cite{CFT18, ConwayBlanchfield}.

\subsection{Generalized S-equivalence}
\label{sec:SEquivalence}
In this section, based on computations from \cite{Cim04, DMO21, CSArf}, we describe how two families of generalized Seifert matrices are related, yielding a notion of \emph{generalized S-equivalence}. This notion will play a key role in Section \ref{sec:AlgConcLinks}, when defining algebraic concordance of links.
\medskip

The way two families of generalized Seifert matrices for the same link are related has already been studied, sometimes indirectly, in the litterature, see for instance the computations in the proofs of \cite[Lemma 4]{Cim04} and \cite[Theorem 2.1]{C-F08}, in \cite[Section 3]{DMO21}, or \cite[Appendix A]{CSArf}. In the last reference, computations are performed in the slightly different setting of \emph{totally connected} C-complexes, where surfaces of two different colors in a C-complex are required to intersect non trivially. It is not difficult to see that the computations therein hold in the case where C-complexes are not totally connected, under slight modifications. This leads to Lemma \ref{lemma:SEquivalence} below, which collects four transformations relating generalized Seifert matrices of a fixed colored link. We now present these four transformations, which reflects Lemma \ref{lemma:SEquivalenceCComplex} in terms of generalized Seifert matrices.

Given two signs~$\varepsilon,\sigma\in\{\pm\}$, we write~$\delta(\varepsilon,\sigma)=1$ if~$\varepsilon=\sigma$, and~$\delta(\varepsilon,\sigma)=0$ otherwise.
Also, for~$z$ an element of a basis of a free abelian group~$H\simeq\Z^n$, we let~$\chi(z)$ denote the element in~$\Z^n$ with the~$z$-coordinate equal to~$1$ and all the others equal to~$0$. Let $A:= \{A^\varepsilon\}_{\varepsilon \in \{\pm\}^\mu}$ and $B:= \{B^\varepsilon\}_{\varepsilon \in \{\pm\}^\mu}$ denote families of generalized Seifert matrices.

(T1) For a color $i \in \{1, \ldots, \mu\}$ and a sign $\sigma_i$, the transformation (T1) transforms the family $A$ into the family $B$ with 
$$
B^\varepsilon=\begin{pmatrix}
A^\varepsilon & 0 & \xi(\varepsilon) \\
0 & 0 & \delta(\varepsilon_i,\sigma_i) \\
\xi(-\varepsilon)^{\T} & \delta(\varepsilon_i,-\sigma_i) & \ell(\varepsilon)
\end{pmatrix}\,,
$$
with~$\ell(\varepsilon) = \ell(-\varepsilon) \in\Z$ and~$\xi(\varepsilon)$ some integral vector.

\medskip

(T2) Let us now fix two distinct colors~$i,j\in\{1,\ldots,\mu\}$ and two signs~$\sigma_i, \sigma_j$. The transformation (T2) transforms the family $A$ into the family $B$ where
$$B^\varepsilon=\begin{pmatrix}
A^\varepsilon & 0 & \xi(\varepsilon)\\
0 & 0 & \delta(\varepsilon_i,\sigma_i)\delta(\varepsilon_j,\sigma_j)\\
\xi(-\varepsilon)^{\T} & \delta(\varepsilon_i,-\sigma_i)\delta(\varepsilon_j,-\sigma_j) & \ell(\varepsilon)
\end{pmatrix}\,,
$$
with~$\ell(\varepsilon)=\ell(-\varepsilon)\in\Z$ and~$\xi(\varepsilon)$ some integral vector.

\medskip

(T3) Fix three pairwise distinct colors~$i,j,k\in\{1,\ldots,\mu\}$  and three signs~$\sigma_i, \sigma_j, \sigma_k$. The transformation (T3) transforms the family $A$ into the family $B$ where
\[
A^\varepsilon=\begin{pmatrix}
A(\varepsilon) & \delta(\varepsilon_i,-\sigma_i)\chi(z) & \xi(\varepsilon)\\
\delta(\varepsilon_i,\sigma_i)\chi(z)^{\T} & 0 & \delta(\varepsilon_i,\sigma_i)\delta(\varepsilon_j,\sigma_j)\\
\xi(-\varepsilon)^{\T} & \delta(\varepsilon_i,-\sigma_i)\delta(\varepsilon_j,-\sigma_j) & \ell(\varepsilon)
\end{pmatrix}
\]
and
\[
B^\varepsilon=\begin{pmatrix}
A(\varepsilon) & \delta(\varepsilon_k,-\sigma_k)\chi(z) & \xi(\varepsilon)- \delta(\varepsilon_i, \sigma_i) \delta(\varepsilon_k, -\sigma_k) \chi(z)\\
\delta(\varepsilon_k,\sigma_k)\chi(z)^{\T} & 0 & \delta(\varepsilon_j,\sigma_j)\delta(\varepsilon_k,\sigma_k)\\
\xi(-\varepsilon)^{\T}- \delta(\varepsilon_i, -\sigma_i)\delta(\varepsilon_k, \sigma_k) \chi(z) & \delta(\varepsilon_j,-\sigma_j)\delta(\varepsilon_k,-\sigma_k) & \ell'(\varepsilon)
\end{pmatrix}\,,
\]
with $\ell(\varepsilon)$ an integer satisfying $\ell(\varepsilon) = \ell(-\varepsilon)$ and 
$$\ell'(\varepsilon) = \ell(\varepsilon)- \delta(\varepsilon_i, -\sigma_i) \delta(\varepsilon_j, \sigma_j)\delta(\varepsilon_k,\sigma_k) - \delta(\varepsilon_i, \sigma_i)\delta(\varepsilon_j, -\sigma_j)\delta(\varepsilon_k, -\sigma_k).$$
The notation $\xi(\varepsilon)$ stands for an integral vector and~$A(\varepsilon)$ for an integral matrix.

\medskip

(T4) Finally, fix two distinct colors~$i,j\in\{1,\ldots,\mu\}$ and two signs~$\sigma_i, \sigma_j$. The transformation (T4) transforms the family $A$ into the family $B$ where
\begin{equation}
\label{eq:T4}
A^\varepsilon=\begin{pmatrix}
A(\varepsilon) & \delta(\varepsilon_i,-\sigma_i)\chi(z) & \xi(\varepsilon)\\
\delta(\varepsilon_i,\sigma_i)\chi(z)^{\T} & 0 & \delta(\varepsilon_i,\sigma_i)\delta(\varepsilon_j,\sigma_j)\\
\xi(-\varepsilon)^{\T} & \delta(\varepsilon_i,-\sigma_i)\delta(\varepsilon_j,-\sigma_j) & \ell(\varepsilon)
\end{pmatrix}
\end{equation}
and
\begin{equation}
\label{eq:T4'}
B^\varepsilon=\begin{pmatrix}
A(\varepsilon) & \delta(\varepsilon_i,\mp\sigma_i)\chi(z) & \xi(\varepsilon)+n\chi(z)\\
\delta(\varepsilon_i,\pm\sigma_i)\chi(z)^{\T} & 0 & \delta(\varepsilon_i,\pm\sigma_i)\delta(\varepsilon_j,\sigma_j)\\
\xi(-\varepsilon)^{\T}+n\chi(z)^{\T} & \delta(\varepsilon_i,\mp\sigma_i)\delta(\varepsilon_j,-\sigma_j) & \ell(\varepsilon)+n
\end{pmatrix}\,,
\end{equation}
with~$\ell(\varepsilon) = \ell(-\varepsilon) \in\mathbb{Z}$,~$\xi(\varepsilon)$ an integral vector,~$A(\varepsilon)$ an integral matrix, and~$n\in \{0, \pm1\}$ independent of~$\varepsilon$.

\medskip

Note that even though these four transformations come from the topological framework of C-complexes, they do make sense in a more general algebraic setting. Indeed, we can apply them to any family of integer square matrices in $\mathcal{M}_\mu$. 

In addition to the four transformations above, we need to consider change of bases, and this corresponds to unimodular congruence. The topological meaning of the following definition is clear from Lemma \ref{lemma:SEquivalence} below.

\begin{definition}
\label{def:GSEquivalence}
Two families of matrices in $\mathcal{M}_\mu$ are \emph{generalized S-equivalent} if they are related by a finite sequence of unimodular congruences, and transformations (T1) through (T4) or their inverses.
\end{definition}

\begin{remark}
\label{rk:SEquivalence}
For $\mu = 1$, we recover the classical definition of S-equivalence. Indeed, as there are no distinct $i,j \in \{1\}$, only the first transformation (T1) is considered, and this is the elementary transformation generating classical S-equivalence, see \cite[Definition 8.3]{Lickorish}. Moreover, the set $\mathcal{M}_1$ consists of pairs of matrices $\{A, A^{\T}\}$ and two families $\{A, A^{\T}\}$ and $\{B, B^{\T}\}$ are congruent if and only if the matrices $A$ and $B$ are congruent as $CAC^{\T} = B$ implies $CA^{\T}C^{\T} = B^{\T}$. 
\end{remark}

The next lemma follows from the computations in \cite[Appendix A]{CSArf}. Details are given in \cite{GSThesis}.

\begin{lemma}
\label{lemma:SEquivalence}
Two families of generalized Seifert matrices for isotopic colored links are generalized S-equivalent.\hfill$\square$
\end{lemma}

\section{Colored algebraic concordance groups}
\label{sec:AlgConcLinks}
This section is devoted to algebraic concordance of families of matrices. It is organized as follows. First, in Section \ref{sub:FamiliesOfMatrices} we define \emph{non degenerate} families of matrices, and show that Witt equivalence yields an equivalence relation on such families. In Section~\ref{sub:AlgebraicConcordanceOfMatrices}, we show that families of matrices with non trivial Alexander polynomial are non degenerate and prove Theorem \ref{thm:IntroAlgConcLinks}.

\subsection{Witt equivalence of families of matrices}
\label{sub:FamiliesOfMatrices}
In this section, we define non degeneracy and Witt equivalence for families of integral matrices. 
\medskip

Recall Definition \ref{def:Metab} of metabolic families of matrices from the introduction. In the classical case of a single matrix, metabolicity leads to an equivalence relation on a set of matrices under certain assumption of non degeneracy. We introduce an analogous notion for families of matrices. As shown in Proposition \ref{prop:WittEquiv}, this notion is sufficient to yield an equivalence relation.

\begin{definition}
\label{def:NonDegeneracy}
A family of integral matrices $A$ in $\mathcal{M}_\mu$ is \emph{non degenerate} if there exists integers $a_\varepsilon \in \Z$ such that the matrix $\sum_{\varepsilon \in \{\pm\}^\mu} a_\varepsilon A^\varepsilon$ has non zero determinant.
\end{definition}

\begin{remark}
\label{rk:NonDegenerate}
\begin{enumerate}[(i)]
\item For $\mu = 1$, elements of $\mathcal{M}_1$ consists of pairs of matrices $\{A, A^{\T}\}$. In this case, the subset of non degenerate families of matrices contains all Seifert matrices of knots, as for such pairs the determinant of $A-A^{\T}$ is equal to one.
\item The definition is equivalent if we allow rational $a_\varepsilon$, as one can always multiply by all denominators.
\end{enumerate}
\end{remark}

Using this notion, we define Witt equivalence for families of non degenerate matrices. Definition \ref{def:Witt} is a particular case, which concerns only families of matrices with non trivial Alexander polynomial. These are non degenerate by Lemma \ref{lemma:APNon0} below.

\begin{definition}
\label{def:WittNdeg}
Two non degenerate families of matrices $A$ and $B$ of $\mathcal{M}_\mu$ are called \emph{Witt equivalent} if the direct sum $A\oplus -B$ is metabolic.
\end{definition}

\begin{remark}
Generalizing an other classical definition of Witt equivalence, it would also be natural to say that two non degenerate families $A$ and $B$ of $\mathcal{M}_\mu$ are Witt equivalent if there exists non degenerate metabolic families $M_1$ and $M_2$ such that $A\oplus M_1$ and $B\oplus M_2$ are congruent. Using Lemma \ref{lemma:WittCancellation} below, one easily shows that this is in fact equivalent to Definition \ref{def:WittNdeg}.
\end{remark} 

The next proposition ensures that Witt equivalence is indeed an equivalence relation on the set of non degenerate families of $\mathcal{M}_\mu$. Its proof follows the classical argument for Seifert matrices, which can be found for example in \cite[Section 3]{Lev69}, see also \cite[Theorem 3.4.4]{LivingstonNaik}.

\begin{proposition}
\label{prop:WittEquiv}
For each $\mu \geq 1$, Witt equivalence is an equivalence relation on the set of non degenerate families of $\mathcal{M}_\mu$.
\end{proposition}

\begin{proof}
Let $A = \{A^\varepsilon\}_{\varepsilon \in \{\pm\}^\mu}$, $B = \{B^\varepsilon\}_{\varepsilon \in \{\pm\}^\mu}$ and $C = \{C^\varepsilon\}_{\varepsilon \in \{\pm\}^\mu}$ be non degenerate families of matrices of size $l$, $m$ and $n$ respectively.

We start with reflexivity. It is clear that after adding the last $l$ rows and columns of $A_\varepsilon \oplus -A_\varepsilon$ to its first $l$ rows and columns respectively, we obtain a matrix with a square of zeros of size $l$ in the top left corner. As this operation is the result of a unimodular congruence and is independent of $\varepsilon$, reflexivity is proved.

Symmetry follows from the obvious facts that the matrices $A_\varepsilon\oplus -B_\varepsilon$ and $B_\varepsilon \oplus -A_\varepsilon$ are congruent, via a matrix independent of $\varepsilon$.

Transitivity is the key point. First, observe that the direct sum of metabolic families of matrices is metabolic. Assume then that $A$ and $B$, respectively $B$ and $C$, are Witt equivalent. It follows that the congruent familes of matrices 
$$(A\oplus -B) \oplus (B\oplus -C) \sim (A\oplus -C) \oplus (B\oplus -B)$$ are metabolic. As $B\oplus -B$ and $(A\oplus -B)\oplus(B\oplus -C)$ are metabolic, Lemma \ref{lemma:WittCancellation} below implies that the family $A\oplus-C$ is metabolic.
\end{proof}

The next lemma is an adaptation of Witt cancellation for families of non degenerate matrices. Its proof follows the original one from \cite[Lemma 1]{Lev69}, which is detailed in \cite[Lemma 3.4.5]{LivingstonNaik}.

\begin{lemma}
\label{lemma:WittCancellation}
Let $A$ and $N$ be families of matrices of $\mathcal{M}_\mu$. Suppose that $N$ and $A\oplus N$ are metabolic and that $N$ is non degenerate. Then $A$ is metabolic.
\end{lemma}

\begin{proof}
Up to congruence, we assume that
$$N^\varepsilon = \begin{pmatrix}0&N_1^\varepsilon\\N_2^\varepsilon&N_3^\varepsilon\end{pmatrix}.$$
Since $A\oplus N$ and $N$ are metabolic, the matrices $A^\varepsilon$ and $N^\varepsilon$ must have even size, say $2m$ and $2k$. Let $B$ denote the direct sum $A\oplus N$. This represents a family of bilinear forms on a module $V_{A} \oplus V_{N} = \Z^{2m} \oplus \Z^{2k}$. As $N$ is metabolic, the second summand splits as $V_{N} = H \oplus V_k$, where $H = \Z^k$ is a metabolizer for $N^\varepsilon$ for all $\varepsilon$ and $V_k = \Z^k$. Using these notations, the metabolic form $B^\varepsilon$ is represented on $V:=V_A \oplus H \oplus V_k$  by the matrix
$$B^\varepsilon = \begin{pmatrix}A^\varepsilon&0&0\\0&0&N_1^\varepsilon\\0&N_2^\varepsilon&N_3^\varepsilon\end{pmatrix}.$$

Since $B$ itself is metabolic, the matrices $B^\varepsilon$ admit a metabolizer independent of $\varepsilon$. For $1 \leq i \leq n := m+k$, let $\alpha_i = (x_i, y_i, z_i) \in V_A \oplus H \oplus V_k$ form a $\Z$-basis for this metabolizer. In order to consider manipulations on these vectors, it may be helpful to arrange them as the columns of an integer $2n \times n$ matrix
$$\alpha = \begin{pmatrix} x_1 & \ldots & x_n \\y_1 & \ldots & y_n \\ z_1 & \ldots & z_n \end{pmatrix}.$$
Note that as these column vectors form a basis for a metabolizer, they are linearly independent, and their family can be completed in a basis of $V$. We also emphasize the fact that all of these vectors are independent of $\varepsilon$. By the definitions, we have $B^\varepsilon(\alpha_i, \alpha_j) = N^\varepsilon((y_i,0), (y_j,0)) = 0$ for all $1 \leq i,j \leq n$ and all $\varepsilon$.

Our next goal is to simplify the matrix $\alpha$. For this, we use the following fact, which appears as statement $(1)$ in \cite[Proof of Lemma 1]{Lev69}. 

\medskip
\textit{Fact.}
If $\gamma_1, \ldots, \gamma_n \in \Z^m$, then there exists a non singular matrix $P = (p_{ij})$ such that if $\gamma_i' = \sum_j p_{ij}\gamma_j$, then $\gamma_1', \ldots, \gamma_r'$ are linearly independent, while $\gamma_{r+1}', \ldots, \gamma_n' = 0$ for some $r$.
\medskip

Using this fact for $z_1,  \ldots, z_n \in \Z^k$, and multiplying $\alpha$ on the right by a corresponding matrix $P$, we get a matrix
$$\alpha' = \begin{pmatrix}x_1' & \ldots & x_r' &x'_{r+1} & \ldots & x_n'\\y_1' & \ldots & y_r' &y'_{r+1} & \ldots & y_n'\\z_1' & \ldots & z_r' & 0 & \ldots & 0\end{pmatrix},$$
where the vectors $z_1', \ldots, z_r' \in V_k$ are linearly independent, and $r \leq k$ is a positive integer. As the matrix $P$ is unimodular, the columns of $\alpha'$ still form a $\Z$-basis of a metabolizer of $B^\varepsilon$.

We then use the fact a second time on the vectors $x_{r+1}', \ldots, x_n'$, yielding a matrix
$$\alpha'' = \begin{pmatrix}x_1' & \ldots & x_r' &x''_{r+1} & \ldots & x''_{r+s} & 0 & \ldots & 0\\y_1' & \ldots & y_r' & y''_{r+1} & \ldots & y''_{r+s} & y''_{r+s+1} & \ldots & y''_n \\ z'_1 & \ldots & z'_r & 0 & \ldots & 0 & 0 & \ldots & 0\end{pmatrix},$$
where the vectors $x''_{r+1}, \ldots, x''_{r+s} \in V_A$ are linearly independent, and $s \leq n-r$ is a positive integer. Again, as this operation involves a unimodular matrix, the columns of $\alpha''$ still form a $\Z$-basis of a metabolizer of $B^\varepsilon$.

We denote by $\alpha''_1, \ldots, \alpha''_n$ the columns of $\alpha''$. As these span a metabolizer for $B^\varepsilon$, for $r+1 \leq i,j \leq r+s$, we have
$$0 = B^\varepsilon(\alpha''_i, \alpha''_j) = A^\varepsilon(x''_i, x''_j) + N^\varepsilon((y''_i,0), (y''_j,0)) = A^\varepsilon(x''_i, x''_j).$$
Thus, $x''_{r+1}, \ldots, x''_{r+s}$ span a subspace of $V_A$ on which $A^\varepsilon$ vanishes for all $\varepsilon$. To prove the lemma, it is therefore sufficient to show that $s \geq m$.
\smallskip

For this, observe that for $r+s+1 \leq i \leq n$, the vectors $\alpha''_i$ satisfy $\alpha''_i = (0,y''_i,0)$.  As the family of the $\alpha''_i$ can be completed in a $\Z$-basis of $V = V_A \oplus H \oplus V_k$, the set of vectors $\{y''_{r+s+1}, \ldots, y''_n\} \subset H$ can be completed in a $\Z$-basis of $H$. Similarly, the set $\{z'_1, \ldots, z'_r\}$ can be completed into a $\Z$-basis of $V_k$.

Furthermore, for $i\geq r+s+1$ and all $j$, independently of $\varepsilon$ we have 
$$0 = B^\varepsilon(\alpha_i'', \alpha_j'') = A^\varepsilon(0,\ast) + N^\varepsilon((y''_i,0), (\ast,z'_j)) = N^\varepsilon((y''_i,0),(0,z'_j)).$$
In particular, the non degenerate form $\sum_{\varepsilon \in I} a_\varepsilon N^\varepsilon$ vanishes on these pairs.

Therefore, with respect to the basis of $H \oplus V_k$ obtained by completing the families $\{y''_{r+s+1}, \ldots, y''_n\}$ and $\{z'_1, \ldots, z'_r\}$, the non degenerate form $\sum_{\varepsilon \in I} a_\varepsilon N^\varepsilon$ is represented by a block matrix of the following shape
$$\begin{pmatrix}
0&0&0&D\\
0&0&\ast&\ast\\
\ast&\ast&\ast&\ast\\
\ast&\ast&\ast&\ast
\end{pmatrix},$$
with $D$ a matrix of size $(n-r-s) \times (k-r)$. As this block matrix represents a non degenerate form, its rows must be linearly independent. Thus the rows of $D$ must be linearly independent. It follows that $n-r-s \leq k-r$, which is equivalent to $m \leq s$.
\end{proof}

\subsection{Algebraic concordance of families of matrices}
\label{sub:AlgebraicConcordanceOfMatrices}
In this section, we show that the property of having non trivial Alexander polynomial is stable by generalized S-equivalence, and implies non degeneracy. After providing examples, we prove Theorem \ref{thm:IntroAlgConcLinks}. We conclude with a couple of remarks on Definition~\ref{def:AlgConcGroup}
\medskip

Before we proceed with algebraic concordance of families of matrices, let us recall that in the case of knots, classical algebraic concordance does not depend on the choice of a Seifert surface. This is because Seifert matrices coming from Seifert surfaces for the same knot are Witt equivalent. In other terms, classical S-equivalence implies Witt equivalence. We provide an example showing that the analogue does not hold for colored links, exhibiting two generalized Seifert matrices for the Hopf link which are not Witt equivalent. 

\begin{example}
\label{example:Hopf}
Let $A = \{A^\varepsilon\}_{\varepsilon \in \{\pm\}^2}$ denote a family of generalized Seifert matrices corresponding to the C-complex for the Hopf link shown in Figure \ref{fig:Counterexample} (an explicit family is given in \cite[Example 1.10]{CCS26}). A direct computation shows that the matrix $A^{--} + (A^{--})^{\T}$ is non degenerate, and has non zero signature. It follows that it is not metabolic, and therefore the matrix $A^{--}$ is not metabolic.

Note that the existence of a contractible C-complex for the Hopf link implies that the family $E$ of empty matrices is also a family of generalized Seifert matrices for the Hopf link. As $A^{--}$ is not metabolic, the direct sum $A \oplus -E = A$ is not metabolic, thus the families $A$ and $E$ are not Witt equivalent.

\begin{figure}[h]
\centering
\begin{overpic}[width=6cm]{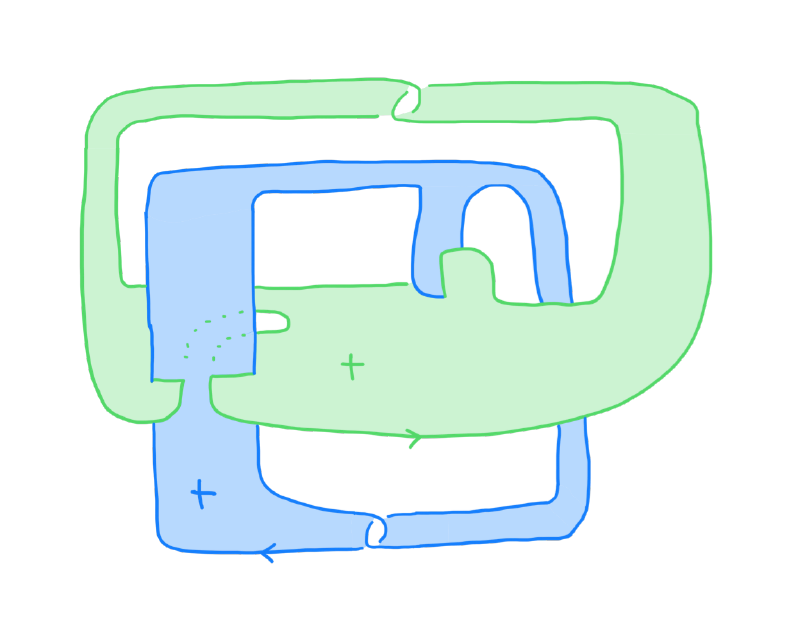}\end{overpic}
\caption{A C-complex for the Hopf link. Any generalized Seifert matrix $A^{--}$ representing the form $\alpha^{--}$ for this C-complex is not metabolic.}
\label{fig:Counterexample}
\end{figure}
\end{example}

Even though families of generalized Seifert matrices for a same colored link are not necessarily Witt equivalent, Lemma \ref{lemma:SEquivalence} ensures that they are generalized S-equivalent. Thus, in order to get an equivalence relation that does not depend on the C-complex, we will include generalized S-equivalence in the definition of algebraic concordance. However, it is not clear right now if non degeneracy of a family is stable by generalized S-equivalence or not.

To circumvent this issue, we restrict to the set $\mathcal{M}_\mu^\ast$ of families of matrices with non trivial Alexander polynomial (recall Definition \ref{def:MAP}). Before we show that this condition implies non degeneracy and only depends on the generalized S-equivalence class of a family of matrices, we make a couple of useful remarks.

\begin{remark}
\label{rk:APnon0Knots}
\begin{enumerate}[(i)]
\item The Alexander polynomial is multiplicative with respect to the direct sum. Indeed, the equality $\Delta_{A \oplus B}(t) = \Delta_A(t) \Delta_B(t)$ follows from the definition.
\item The Alexander polynomial is a symmetric element of $\Lambda$, in the sense that $\Delta_A(t^{-1}) \,\dot=\, \Delta_A(t)$. This is immediate from the definition.
\end{enumerate}
\end{remark}

The next lemma allows to apply Witt cancellation to families of matrices with non trivial Alexander polynomial.

\begin{lemma}
\label{lemma:APNon0}
A family of matrices with non trivial Alexander polynomial is non degenerate. 
\end{lemma}

\begin{proof}
Let $A = \{A^\varepsilon\vert\; \varepsilon \in \{\pm\}^\mu\}$.
If this family has non trivial Alexander polynomial, the polynomial 
$$P(t_1, \ldots, t_\mu) := \det{\sum_{\varepsilon \in \{\pm\}^\mu}\varepsilon_1 \cdots \varepsilon_\mu t_1^{\frac{1-\varepsilon_1}{2}} \cdots t_\mu^{\frac{1-\varepsilon_\mu}{2}}A^\varepsilon}$$
is non trivial. Therefore, there exists rational numbers $x_1, \ldots, x_\mu$ such that $P(x_1, \ldots, x_\mu)$ is nonzero. Setting $a_\varepsilon = \varepsilon_1 \cdots \varepsilon_\mu x_1^{\frac{1-\varepsilon_1}{2}} \cdots x_\mu^{\frac{1-\varepsilon_\mu}{2}}$ we get $$\det{\sum_{\varepsilon \in \{\pm\}^\mu} a_\varepsilon A^\varepsilon} = P(x_1, \ldots, x_\mu) \neq 0.$$
It then follows from the second item of Remark \ref{rk:NonDegenerate} that the family $A$ is non degenerate.
\end{proof}

We now show that the property of having non trivial Alexander polynomial is invariant by generalized S-equivalence. Before we do so, let us recall that the Alexander polynomial of a link is well defined only up to multiplication by units of $\Lambda$. In our case, when dealing only with generalized Seifert matrices and not with C-complexes, we loose some information. Indeed, in order to compute the Alexander polynomial of a colored link, one needs the generalized Seifert matrices together with an information concerning the genus of the C-complex, and of its underlying colored surfaces, see \cite{Cim04}. Even though this information is encoded in the size of the matrix for $\mu = 1$, this is not the case for $\mu \geq 2$. In this case, without this information, the generalized Seifert matrices allow to compute the Alexander polynomial only up to multiplication by elements of the form $(t_i-1)$. This reflects through generalized S-equivalence in the following sense : for $\mu \geq 2$, the Alexander polynomial of a family of matrices is invariant by generalized S-equivalence if and only if it is considered up to multiplication by units of $\Lambda_S$. Indeed, this follows readily from the explicit formula for the main result of \cite{Cim04}.

\begin{lemma}
\label{lemma:APInvariance}
If $A$ and $B$ in $\mathcal{M}_\mu$ are related by generalized S-equivalence, then $\Delta_A(t)\, \ddot = \, \Delta_B(t)$. In particular, $A$ has non trivial Alexander polynomial if and only if $B$ has, and in this case, the direct sum $A \oplus B$ also has non trivial Alexander polynomial.
\end{lemma}

\begin{proof}
The second part of the lemma is immediate from the first one, and the last part follows from the multiplicativity of the Alexander polynomial, see Remark \ref{rk:APnon0Knots} (i).

To prove the first part, one needs to check that the Alexander polynomial is unchanged by the four moves of generalized S-equivalence. This has been done by Cimasoni in \cite{Cim04} for the first three moves and completed by \cite{DMO21} for the last move. Note that they work with a normalized version of the Alexander polynomial including some genus information, but this only involves factor $(t_i-t_i^{-1})$, which are invertible elements of $\Lambda_S$. Hence, their computations show that under generalized S-equivalence, the Alexander polynomial only changes by invertible elements of $\Lambda_S$. Details are given in \cite{GSThesis}.
\end{proof}

\begin{remark}
\begin{enumerate}[(i)]
\item Even though the condition of having non trivial Alexander polynomial seems to be a more restrictive condition than being a non degenerate family, we have at the moment no examples of such families coming from links, i.e. a family of generalized Seifert matrices which is non degenerate but does have trivial Alexander polynomial.
\item As the family of links having non degenerate generalized Seifert matrices is \emph{a priori} bigger than the family of links having non trivial Alexander polynomial, it might be tempting to work with non degenerate families of matrices, instead of families of matrices with non trivial Alexander polynomial. We give two reasons why this is in fact not straightforward. First, as already mentioned, it is not clear that the property of being non degenerate for a family of matrices is stable by S-equivalence. Therefore, this is a property that might depend on the choice of a C-complex. Second, we use the fact that we consider links with non trivial Alexander polynomial in the proof of Theorem \ref{prop:IntroAlgConcLinks}.
\end{enumerate}
\end{remark}

Before we proceed with the proof of Theorem \ref{thm:IntroAlgConcLinks}, we show that there exist links admitting degenerate families of generalized Seifert matrices, by providing two examples. Note that both have trivial Alexander polynomial, which is coherent with Lemma \ref{lemma:APNon0}.

\begin{example}
\label{example:LinkDegenerateFamily}
\begin{enumerate}[(i)]
\item For any integer $c>1$, and any $\mu \leq c$, the $\mu$-colored $c$-component unlink admits a degenerate family of generalized Seifert matrices. Indeed, the C-complex shown in Figure \ref{fig:DegenerateUnlink} shows that the family of $c-1 \times c-1$ zero matrices arises as a family of generalized Seifert matrices for the unlink. This family is clearly degenerate.

\begin{figure}[h]
\centering
\begin{overpic}[width=12cm]{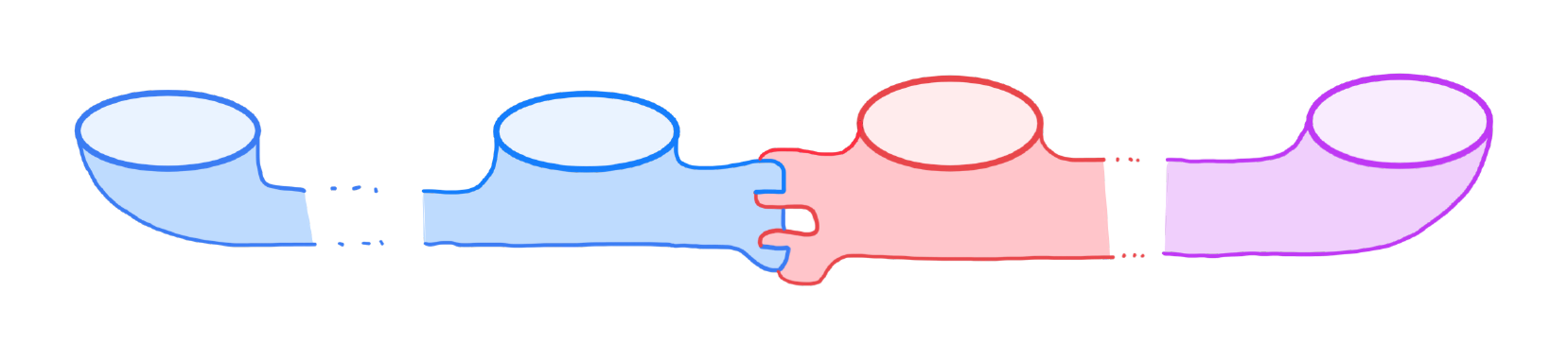}\end{overpic}
\caption{A C-complex for the $c$-component, $\mu$-colored unlink, yielding the family of zero matrices as generalized Seifert matrices.}
\label{fig:DegenerateUnlink}
\end{figure}

\item As a non trivial example, the link $L11n244$, shown in the left of Figure \ref{fig:DegenerateNonTrivial} admits a degenerate family of matrices as a family of generalized Seifert matrices. To see this, consider the C-complex shown in the right of Figure \ref{fig:DegenerateNonTrivial}, with a basis of the first homology group depicted therein. The corresponding generalized Seifert matrices are given by 
$$A^{++} = \begin{pmatrix}1&0&1\\0&1&0\\1&0&1\end{pmatrix} \text{   and   } A^{+-} = \begin{pmatrix}1&1&1\\ 0&1&0\\1&1&1\end{pmatrix},$$
and their transpose. Note that each of these four matrices has linearly dependent first and last column. It follows that any linear combination of these four matrices also has linearly dependent first and last column, and the family is therefore degenerate.

\begin{figure}[h]
\centering
\begin{overpic}[width=12cm]{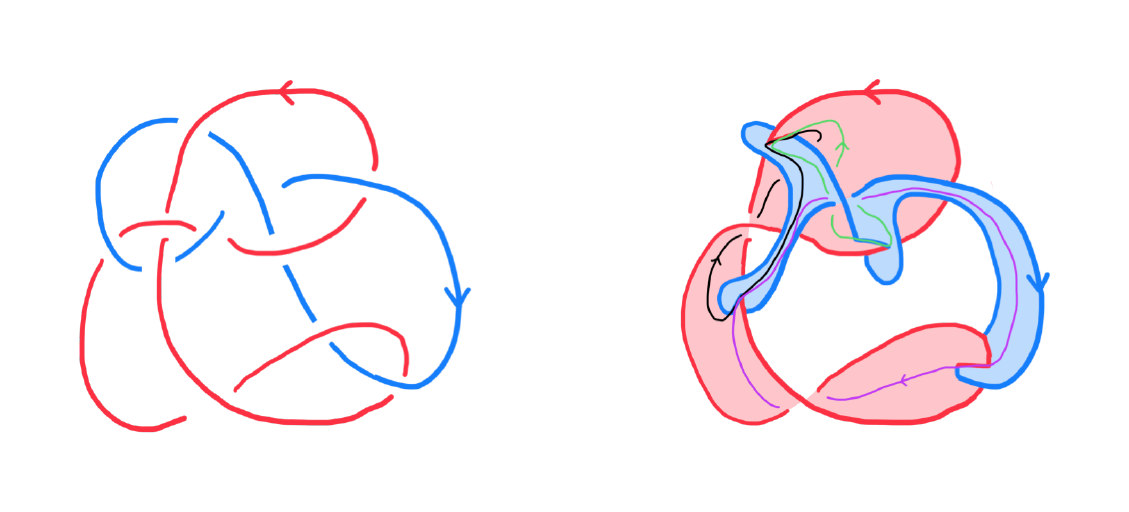}\end{overpic}
\caption{The link $L11n244$, shown on the left admits the C-complex shown on the right. Computing the associated generalized Seifert matrices with the three curves shown therein yields a degenerate family of matrices.}
\label{fig:DegenerateNonTrivial}
\end{figure}

\end{enumerate}
\end{example}

We now prove Theorem \ref{thm:IntroAlgConcLinks}.

\begin{proof}
The fact that algebraic concordance is an equivalence relation is immediate since Witt equivalence and generalized S-equivalence are equivalence relations.

Furthermore, direct sum yields a well defined operation on the quotient, as it preserves generalized S-equivalence and Witt equivalence. Indeed, if $A\sim A'$ and $B \sim B'$ are Witt equivalent (respectively generalized S-equivalent) families of matrices, then $A\oplus B$ and $A' \oplus B'$ are Witt equivalent (respectively generalized S-equivalent). This operation is closed as direct sum preserves the property of having non trivial Alexander polynomial, by Lemma \ref{lemma:APNon0}.

It remains to verify that this well defined operation yields an abelian group structure. Associativity and commutativity are immediate as direct sum is associative and commutative up to congruence (and congruent matrices are generalized S-equivalent).

The zero element is given by the class of metabolic families of matrices, which contains the family of empty matrices. Indeed, any metabolic family of matrices is Witt equivalent to the family of empty matrices, and taking direct sum with a family of empty matrices does not change an element.

Finally, the inverse of a class of family of matrices $A$ is the class of the family $-A$, as $A \oplus -A$ is metabolic (this follows from the reflexivity of Witt equivalence, see Proposition \ref{prop:WittEquiv}). It is easy to check that $-A$ has non trivial Alexander polynomial if $A$ has, hence the inverse belongs to the group.
\end{proof}

Recall from Definition \ref{def:AlgConcGroup} that the group of Theorem \ref{thm:IntroAlgConcLinks} is called the $\mu$-colored algebraic concordance group, and is denoted by $\mathcal{C}_\mu^\alg$. We conclude this section with a couple of remarks on this definition.

\begin{remark}
\label{rk:AlgConcGroupBig}
\begin{enumerate}[(i)]
\item In the classical case of knots, the algebraic concordance group consists of the classes of all Seifert matrices of knots, i.e. integral square matrices $A$ satisfying $\det(A~-~A^{\T})~=~1$. Our group $\mathcal{C}_\mu^\alg$ is a priori \emph{bigger} in the sense that it contains the classes of all families of matrices in $\mathcal{M}_\mu^\ast$, some of which might not be realizable as families of generalized Seifert matrices. It is not difficult to see that, since we consider colored links, the classes of families of generalized Seifert matrices form a subgroup of $\mathcal{C}_\mu^\alg$. This is one of our main motivations for the consideration of colored links. However, elements of this subgroup seem hard to characterize algebraically, as we have no criterion such as in the case of knots. Therefore, the consideration of $\mathcal{C}_\mu^\alg$ seems the most natural. Let us also mention that all concordance invariants of colored links we will consider in Section \ref{sec:AlgConcInvariants} will straightforwardly extend to families of matrices in $\mathcal{M}_\mu^\ast$.
\item Note that the group $\mathcal{C}_1^{\alg}$ contains the classes of all classical Seifert matrices of knots, but also the classes of all Seifert matrices of one colored links with non trivial Alexander polynomial. By Remark \ref{rk:AlgConcKnots2} below, which ensures that two pairs of Seifert matrices coming from knots are algebraically concordant as pairs if and only if the underlying matrices are algebraically concordant, the inclusion of pairs of Seifert matrices of knots into $\mathcal{M}^{\ast}_1$ induces an injective homomorphism from the algebraic concordance group of knots into $\mathcal{C}_1^{\alg}.$ This homomorphism is not surjective by Remark \ref{rk:AlgConcGroupMu=1}.
\end{enumerate}
\end{remark}

\section{Smoothly concordant links are algebraically concordant}
\label{sub:AlgebraicConcordanceOfLinks}
In this section, we define algebraic concordance of colored links and prove Theorem \ref{prop:IntroAlgConcLinks}. This leads to a well defined map from the set of $\mu$-colored links up to smooth concordance into the $\mu$-colored algebraic concordance group.
\medskip

\begin{definition}
\label{def:AlgConcLinks}
Two colored links with non trivial Alexander polynomial are \emph{algebraically concordant} if any two of their families of generalized Seifert matrices are algebraically concordant.
\end{definition}

\begin{remark}
\label{rk:AlgConcKnots2}
Two knots are algebraically concordant as $1$-colored links if and only if any two pairs of their Seifert matrices are algebraically concordant. By Remark \ref{rk:AlgConcKnots1}, this is the case if and only if any two of their Seifert matrices are algebraically concordant, which is the classical definition of algebraic concordance of knots. Nevertheless, the case $\mu = 1$ is a more general definition than classical algebraic concordance of knots, as it allows the consideration of $1$-colored links with more than one component.
\end{remark}

In the classical case of knots, this terminology is motivated by the fact that concordant knots are algebraically concordant. We now prove that the analogue holds for colored links, at least in the smooth category. We say that that two colored links are (smoothly) \emph{concordant} if they are (smoothly) concordant as links via a concordance preserving the colors. With Definition \ref{def:AlgConcLinks}, Theorem \ref{prop:IntroAlgConcLinks} reformulates as follows.

\begin{theorem}
\label{prop:AlgConcLinks}
Smoothly concordant colored links with non trivial Alexander polynomial are algebraically concordant.
\end{theorem}

\begin{remark}
\label{remark:AlgConcGroup}
Theorem \ref{prop:AlgConcLinks} ensures the existence of a well defined map $\varphi$ from the set of $\mu$-colored links up to smooth concordance into $\mathcal{C}_\mu^{\alg}$, assigning to the concordance class of a colored link the algebraic concordance class of its family of generalized Seifert matrices. This map is \emph{additive} in the following sense. Let $L$ and $L'$ be $\mu$-colored links, and perform a connected sum $L \# L'$ along any two components of $L$ and $L'$ of the same color. Then 
$$\varphi(L\#L') = \varphi(L) \oplus \varphi(L').$$
Note that this property is valid even though the connected sum of two links is not a well defined operation, as it depends on the choice of the components.
\end{remark}

The remaining part of this section is devoted to the proof of Theorem \ref{prop:AlgConcLinks}.

\begin{proof}[Proof of Theorem \ref{prop:AlgConcLinks}.]
The proof mostly follows an argument of \cite[Section 4]{Coo82}. There are two main differences. First, we consider links with arbitrary numbers of colors and components, while Cooper restricts to two component, $2$-colored links. Second, we deal with concordance between any two links, while Cooper restricts to slice links, i.e. links concordant to the two component unlink. 

Let $L$ and $L'$ be smoothly concordant colored links, with non trivial Alexander polynomials. Note that by definition, $L$ and $L'$ have the same number of components of each color. Let $F$ and $F'$ be C-complexes for $L$ and $L'$ respectively and let $A$ and $A'$ be associated families of generalized Seifert matrices. We show that there exists a colored link $L_{1/2}$, having non trivial Alexander polynomial, with the same number of components of each color as $L$ and $L'$, and with C-complexes $F_{1/2}$ and $F_{1/2}'$ such that the associated families of generalized Seifert matrices $A_{1/2}$ and $A'_{1/2}$ satisfy $A \sim_\textit{Witt} A_{1/2}$ and $A'\sim_\textit{Witt}A'_{1/2}$. The fact that $L$ and $L'$ are algebraically concordant then follows as 
\begin{equation}
\label{eq:AlgConc}
A \sim_\textit{Witt} A_{1/2} \sim_{S} A'_{1/2} \sim_\textit{Witt} A'.
\end{equation}

The link $L_{1/2}$ is obtained from $L$ and $L'$ by \emph{band summing an unlink}, a notion that we recall below. Note that the existence of such a link for concordant links is well known, see for instance \cite[Lemma 1]{NakagawaAP}. We chose to recall a proof here in order to deal with colors, and for the sake of completeness.

To build the link $L_{1/2}$, we use Morse theory. As the concordance is smooth, it can be arranged in the following way : starting from $L$, the first critical points are the minima, then come all saddles connecting these minima to the cylinders of the concordance, and after that all remaining saddles. The maxima come last. This is a consequence of the fact that a smooth manifold can be decomposed into handles attached in increasing order, see for instance \cite[Proposition 4.2.7]{GompfStipsicz}.

It follows that there exists a colored link $L_{1/2}$, obtained from both $L$ and $L'$ by band summing an unlink. More precisely, $L_{1/2}$ is obtained from $L$ as follows. Start with $L$, and take the disjoint union with an unlink $U$, unlinked from $L$ (this corresponds to minima). Choose arcs running from the components of $L$ to $U$, disjoint from each other and with interiors disjoint from $L$ and $U$. Require that each component $U$ is the endpoint of exactly one arc. Then perform band sums along these arcs, that is remove a neighborhood of each arc in $L$ and $U$, and glue two parallel copies of the arc (these correspond to saddles). The result is a link with the same number of components as $L$, and we denote it by $L_{1/2}$. It is colored as follows : each component of $U$ inherits the color of the component of $L$ with which it is band summed. Note that this preserves the number of components in each color, as no new component is added, no component is deleted, and no two components are merged. This process is illustrated in Figure \ref{fig:BandSumUnlink}.

Turning the concordance upside down, and starting from $L'$ shows that $L_{1/2}$ is also obtained from $L'$ by band summing an unlink. Observe that the colorings of $L_{1/2}$ obtained from $L$ and $L'$ are the same. Indeed, $L_{1/2}$ is obtained as a slice of the concordance, and its coloring correspond to the colors of the cylinders. These preserve the colors of $L$ and $L'$ by definition.

\begin{figure}[h]
\centering
\begin{overpic}[width=14cm]{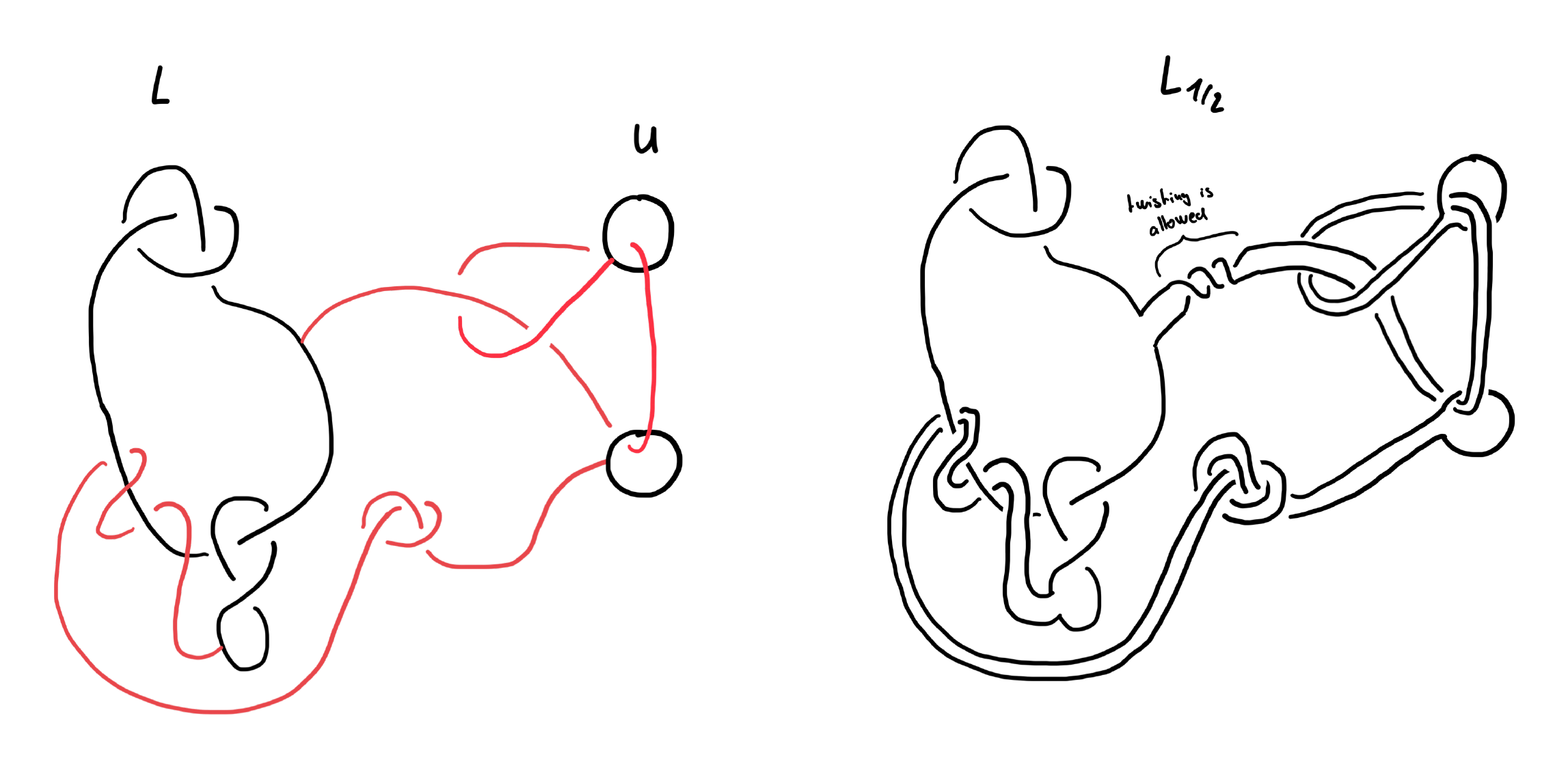}\end{overpic}
\caption{On the left, a link $L$ with a two components unlink, in black. In red, the arc along which the band sum is performed. On the right, the result after band sum.}
\label{fig:BandSumUnlink}
\end{figure}

A C-complex $F_{1/2}$ for $L_{1/2}$ is obtained from $F$ as follows. First, fill the components of $U$ with discs, and connect these to $F$ using the bands of the band sum. The result is a priori not a C-complex, as the bands might intersect $F$ and the discs. Observe that each intersection consists of a band intersecting a surface, and can be resolved as indicated in Figure \ref{fig:ResolvingBandIntersections}, depending on whether the intersection happens between surfaces of the same or different colors. After resolving all intersections, we end up with a C-complex $F_{1/2}$ for $L_{1/2}$.

\begin{figure}[h]
\centering
\begin{overpic}[width=14cm]{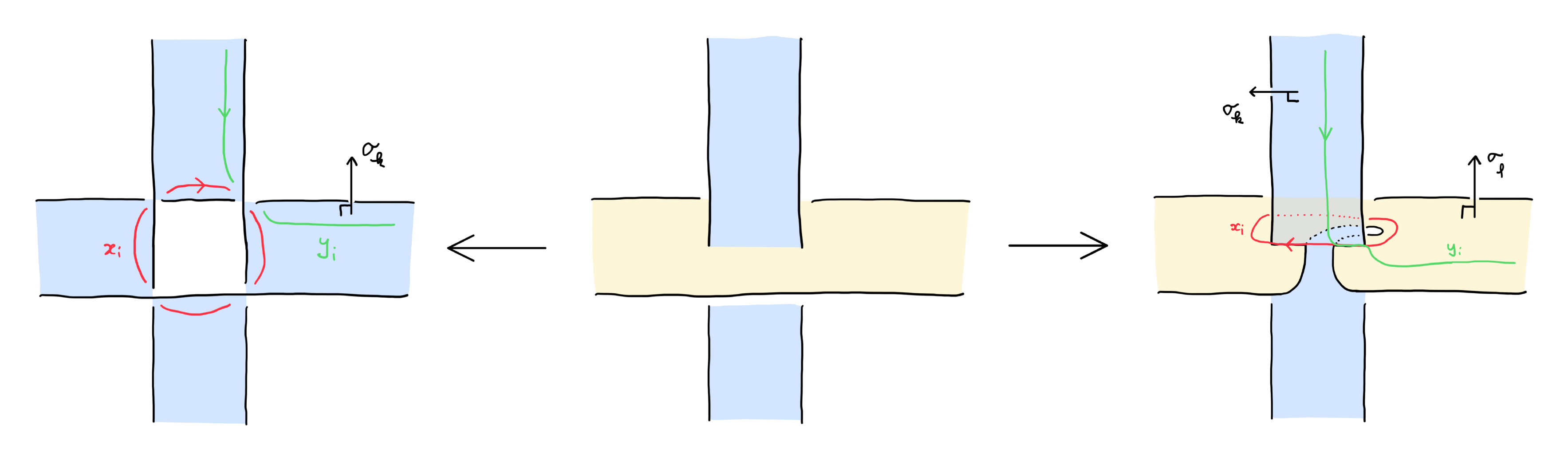}\end{overpic}
\caption{In the middle, a band, in blue, intersecting a surface, in yellow. Left, how the intersection is resolved if both the band and the surface are of the same color. Right, how the intersection is resolved if the band and the surface have different colors.}
\label{fig:ResolvingBandIntersections}
\end{figure}

As the C-complex $F$ is connected, there is a canonical isomorphism 
$$H_1(F_{1/2}) \cong \left(\bigoplus_{i=1}^n \Z[x_i]\right) \oplus \left(\bigoplus_{i=1}^n \Z[y_i]\right) \oplus H_1(F),$$
with $x_i$ and $y_i$ the curves indicated in Figure \ref{fig:ResolvingBandIntersections}, and $n$ the number of resolved intersections. Clearly, for all $i,j$, for all sign $\varepsilon$ and for all $y \in H_1(F)$ we have
$$\alpha^\varepsilon(x_i, x_j) = \alpha^\varepsilon(x_i, y) = 0.$$
Therefore, in a basis corresponding to the above decomposition, the generalized Seifert form $\alpha^\varepsilon$ is represented by a matrix of the shape
$$A_{1/2}^\varepsilon:=\begin{pmatrix}
0&C^\varepsilon & 0\\ D^\varepsilon &\ast&\ast \\0&\ast& A^\varepsilon
\end{pmatrix},$$
for some $n\times n$ matrices $C^\varepsilon$ and $D^\varepsilon$ satisfying $D^\varepsilon = (C^{-\varepsilon})^{\T}$.
Taking the direct sum with $-A^\varepsilon$, there are congruences
$$A_{1/2}^\varepsilon \oplus -A^\varepsilon = \begin{pmatrix}
0&C^\varepsilon & 0&0\\ D^\varepsilon&\ast&\ast&0 \\0&\ast& A^\varepsilon&0\\0&0&0&-A^\varepsilon
\end{pmatrix} \sim
\begin{pmatrix}0&C^\varepsilon & 0&0\\ D^\varepsilon&\ast&\ast&0 \\0&\ast& 0&-A^\varepsilon\\0&0&-A^\varepsilon&-A^\varepsilon\end{pmatrix} 
\sim \begin{pmatrix}0&0&C^\varepsilon&0\\ 0&0&\ast&-A^\varepsilon\\D^\varepsilon&\ast&\ast&0\\0&-A^\varepsilon&0&-A^\varepsilon\end{pmatrix},$$
realized by summing or permuting rows and columns. These congruence do not depend on $\varepsilon$, and the result is a metabolic matrix. Therefore, the family of generalized Seifert matrices $A_{1/2} \oplus -A$ is metabolic, and the families $A_{1/2}$ and $A$ are Witt equivalent.

We delay the proof that the family $A_{1/2}$ has non zero Alexander polynomial to Lemma \ref{lemma:A1/2APNon0} below.

Similarly, starting with the C-complex $F'$ for $L'$, we build a C-complex $F'_{1/2}$ for $L_{1/2}$ with a family of generalized Seifert matrices $A_{1/2}'$ Witt equivalent to $A'$. 

As $A_{1/2}$ and $A_{1/2}'$ are generalized Seifert matrices coming from C-complexes for the same colored link, they are S-equivalent. This establishes the equivalences in (\ref{eq:AlgConc}), concluding the proof of the theorem.
\end{proof}

\begin{lemma}
\label{lemma:A1/2APNon0}
The family $A_{1/2}$ satisfies $\Delta_{A_{1/2}}(t) \,\ddot=\, f(t)f(t^{-1})\Delta_A(t)$ for some Laurent polynomial $f \in \Lambda$ with $\vert f(1, \ldots, 1) \vert = 1$. In particular, the family $A_{1/2}$ has non trivial Alexander polynomial.
\end{lemma}

\begin{proof}
By definition, we have 
$$A_{1/2}(t) = \begin{pmatrix} 0 & C(t) & 0 & \\ D(t) & \ast & \ast \\ 0 & \ast & A(t)\end{pmatrix}.$$
Therefore, 
$$\Delta_{A_{1/2}}(t) = \det(A_{1/2}(t)) = \det(C(t)) \det(D(t)) \Delta_{A}(t).$$
As the family of matrices $D$ satisfy $D^\varepsilon = (C^{-\varepsilon})^{\T}$, we have $\det(D(t)) \, \dot =\, \det(C(t^{-1}))$. Thus, the previous displayed equality reads
$$\Delta_{A_{1/2}}(t) \,\dot=\,\det(C(t)) \det(C(t^{-1})) \Delta_A(t).$$
Therefore, to complete the proof, it remains to show that $\det(C(t))$ is a polynomial of the form 
\begin{equation}
\label{eq:AF1}
\det(C(t)) \,\dot=\, f(t) \cdot \prod(t_i-1)^{n_i}, \text{ with } \vert f(1, \ldots, 1)\vert = 1.
\end{equation}

The remainder of the proof is devoted to establishing $(\ref{eq:AF1})$. This goes through an explicit description of the matrices $C^\varepsilon$, that we now begin.

By definition, $C^\varepsilon_{ij} = \lk(x_i, y_j^\varepsilon)$ where $x_i$ is a cycle coming from a resolution of a band intersection, and for which we do have a complete picture, and $y_j$ is a cycle, also coming from a band intersection, but for which we do not have a complete picture. These are represented in Figure \ref{fig:ResolvingBandIntersections}. 

For $i = j$, a direct computation shows that if the $i^{\text{th}}$ band intersection is resolved as in the left of Figure \ref{fig:ResolvingBandIntersections}, we have 
$$\lk(x_i^\varepsilon, y_i) = \delta(\varepsilon_k, \sigma_k)$$
with $\sigma_k \in \{\pm\}$ the sign indicated in Figure \ref{fig:ResolvingBandIntersections} and $\delta$ the Kronecker delta. Similarly, if the band intersection is resolved as in the right of Figure \ref{fig:ResolvingBandIntersections}, for all $\varepsilon$, we have
$$\lk(x_i^\varepsilon,y_i) = \delta(\varepsilon_k, \sigma_k)\delta(\varepsilon_l, \sigma_l),$$
with $\sigma_k$, $\sigma_l \in \{\pm\}$ the signs indicated in Figure \ref{fig:ResolvingBandIntersections}. 

It remains to compute $\lk(x_i, y_j^\varepsilon)$ for $i \neq j$. The key point is the following : we can always assume that the cycle $y_j$ crosses the resolved band intersection corresponding to $x_i$ either vertically, or horizontally, but never with an \emph{angle}, as in Figure \ref{fig:ResolvingBandIntersections} which corresponds to the case $i=j$. Indeed, we can always assume that $y_j$ follows a band, and goes out of the band only at the resolved intersection corresponding to $x_j$. Therefore, we distinguish two cases, depending on the resolution of the band intersection. If it is resolved as in the left of Figure \ref{fig:ResolvingBandIntersections}, we can assume that the cycles $x_i$ and $y_j$ are disjoint and thus
$$\lk(x_i^\varepsilon, y_j) = \lk(x_i, y_j)$$
is independent of $\varepsilon$. On the other hand, if the intersection is resolved as in the right of Figure \ref{fig:ResolvingBandIntersections}, either $y_j$ intersects $x_i$ geometrically once, yielding $\lk(x_i, y_j^\varepsilon) = \pm\delta(\varepsilon_k, \sigma_k)$, or $y_j$ is disjoint from $x_i$ and in this case $\lk(x_i^\varepsilon, y_j) = 0$. Therefore, using the notation $\langle x_i, y_j\rangle$ for the algebraic intersection number of $x_i$ and $y_j$, we have
$$\lk(x_i^\varepsilon, y_j) = \pm \langle x_i, y_j\rangle \delta(\varepsilon_k, \sigma_k).$$

Without lost of generality, we can assume that the band intersections corresponding to $x_1, \ldots, x_{m-1}$ are between surfaces of the same color, while those corresponding to $x_{m}, \ldots, x_n$ are between surfaces of different colors. Using the explicit computations above, the matrix $C^\varepsilon$ is of the form
$$\footnotesize
\begin{pmatrix} \delta(\varepsilon_{k_1}, \sigma_{k_1}) &\cdots & \lk(x_1, y_{m-1}) & \lk(x_1, y_{m}) & \cdots & \lk(x_1, y_n) \\
\vdots &\ddots & \vdots & \vdots && \vdots \\
\lk(x_{m-1}, y_1) & \cdots & \delta(\varepsilon_{k_{m-1}}, \sigma_{k_{m-1}}) & \lk(x_{m-1}, y_{m}) & \cdots & \lk(x_{m-1}, y_n)\\
\pm \langle x_{m}, y_1\rangle\delta(\varepsilon_{k_{m}}, \sigma_{k_{m}}) & \cdots & \pm \langle x_{m}, y_{m-1}\rangle \delta(\varepsilon_{k_{m}}, \sigma_{k_{m}}) & \delta(\varepsilon_{k_{m}}, \sigma_{k_{m}}) \delta(\varepsilon_{l_{m}}, \sigma_{l_{m}}) & \cdots & \pm \langle x_{m}, y_n\rangle\delta(\varepsilon_{k_{m}}, \sigma_{k_{m}})\\
\vdots & & \vdots & \vdots & \ddots & \vdots\\
\pm \langle x_{n}, y_1 \rangle \delta (\varepsilon_{k_{n}}, \sigma_{k_{n}}) & \cdots & \pm \langle x_{n}, y_{m-1}\rangle \delta(\varepsilon_{k_{n}}, \sigma_{k_{n}}) & \pm \langle x_{n}, y_{m}\rangle\delta(\varepsilon_{k_{n}}, \sigma_{k_{n}})  & \cdots & \delta(\varepsilon_{k_{n}}, \sigma_{k_{n}}) \delta(\varepsilon_{l_{n}}, \sigma_{l_{n}})
\end{pmatrix}$$
Combining this equality together with the definition of $C(t)$, we obtain
$$C(t) = \begin{pmatrix} \sigma_{k_1} t_{k_1}^\frac{1-\sigma_{k_1}}{2} \prod_{i\neq k_1} (1-t_i)  & \ldots & \lk(x_1, y_n) \prod_i (1-t_i)\\
\vdots & \ddots & \vdots \\
\pm x_n\cdot y_1 \sigma_{k_n} t_{k_n}^{\frac{1-\sigma_{k_n}}{2}}\prod_{i\neq k_n} (1-t_i) & \cdots &  \sigma_{k_n}\sigma_{l_n}t_{k_n}^{\frac{1-\sigma_{k_n}}{2}} t_{l_n}^{\frac{1-\sigma_{l_n}}{2}} \prod_{i \neq k_n, l_n} (1-t_i) \end{pmatrix}.$$
Roughly speaking, if the coefficient of $C^\varepsilon_{i,j}$ is independent of $\varepsilon$, then $C(t)_{i,j} = C^\varepsilon_{i,j} \prod_{i}(1-t_i)$. On the other hand, if $C^\varepsilon_{i,j}$ has a factor $\delta(\varepsilon_{k}, \sigma_k)$, the term $(1-t_k)$ is removed from the product, and a coefficient $\sigma_k t_k^{(1-\sigma_k)/2}$ appears. This is an immediate computation from the definition of $C(t)$.

Observe that for $s \in \{1, \ldots, m-1\}$, the row $s$ is divisible by $\prod_{i\neq k_{s}}(1-t_i)$, and for $s \in \{m, \ldots, n\}$, the row $s$ is divisible by $\prod_{i \neq k_s, l_s}(1-t_i)$. 
It follows that the determinant of $C(t)$ is equal, up to multiplications by factors $(1-t_i)$, to the determinant of the matrix
$$C(t)_{\text{red}} := \begin{pmatrix} \sigma_{k_1} t_{k_1}^\frac{1-\sigma_{k_1}}{2}  & \ldots & \lk(x_1, y_n) (1-t_{k_1})\\ 
 \vdots & \ddots & \vdots \\
\pm x_n\cdot y_1 \sigma_{k_n} t_{k_n}^{\frac{1-\sigma_{k_n}}{2}} (1-t_{l_n}) & \cdots &  \sigma_{k_n}\sigma_{l_n}t_{k_n}^{\frac{1-\sigma_{k_n}}{2}} t_{l_n}^{\frac{1-\sigma_{l_n}}{2}} \end{pmatrix}.$$
We set $f(t):= \det(C(t)_{\text{red}})$ and note that, as $C(1, \ldots, 1)_{\text{red}}$ is a diagonal matrix with only $\pm1$ on the diagonal, we have $\vert f(1, \ldots, 1) \vert =~1$. Hence the equality 
$$\det(C(t)) \,\dot=\, \det(C(t)_{\text{red}}) \cdot \prod_i (t_i-1)^{n_i} = f(t) \cdot \prod_i (t_i-1)^{n_i}$$
establishes $(\ref{eq:AF1})$ and completes the proof.
\end{proof}

\section{Algebraic concordance invariants of links}
\label{sec:AlgConcInvariants}
The aim of this section is twofold. First, in Sections \ref{sub:Signature} through \ref{sub:Blanchfield}, we review the following known invariants of colored links : the signature, the Alexander polynomial and the Blanchfield pairing. We show that these invariants naturally extend to families of matrices in $\mathcal{M}_\mu^\ast$, and use them to define algebraic concordance invariants. In Section \ref{sub:Structure}, we show that these invariants can be combined to unveil part of the structure of $\mathcal{C}_\mu^\alg$, proving Theorem \ref{prop:StructureAlgConcGroup}. Finally, Section \ref{sub:WittFinite} briefly discusses the potential existence of a $\Z_4^\infty$ summand of $\mathcal{C}_\mu^\alg$.

\subsection{Signatures}
\label{sub:Signature} In this section, after briefly recalling the Levine-Tristram signature of knots and its extension to colored links, we define the signature of families of matrices in $\mathcal{M}_\mu^\ast$ and show that it is an algebraic concordance invariant. We then apply this result to links, proving Corollary~\ref{cor:Signature}.
\medskip

The Levine-Tristram signature of a knot is a well known algebraic concordance invariant, defined as follows. For a knot $K$, with a pair of Seifert matrices $\{A, A^{\T}\}$ and a complex number $\omega \in S^1\setminus \{1\}$, the signature $\sigma_K(\omega)$ of $K$ at $\omega$ is the signature of the hermitian matrix 
$$H(\omega) := (1-\omega)A + (1-\overline{\omega})A^{\T}.$$
A direct computation shows that the signature is invariant under classical S-equivalence, and therefore is a well defined invariant of $K$. The fact that it is an algebraic concordance invariant (for $\omega$ not a root of the Alexander polynomials) is immediate from the purely algebraic fact that the signature of a non degenerate hermitian matrix is unchanged by Witt equivalence.

The signature of a colored link was defined by Cimasoni and Florens in \cite{C-F08}, using C-complexes and generalized Seifert matrices. Their techniques extended previous work of Cooper, which considered the case of two component, $2$-colored links, see \cite{Coo82}. Similarly as in the classical case of knots, Cimasoni and Florens proved that this signature does not depend on the choice of generalized Seifert matrices by showing that it is invariant under generalized S-equivalence, (see \cite{DMO21} for a completed proof). In their original paper, they also showed that, when restricted to suitable complex numbers, the signature is a topological concordance invariant of colored links. Their proof is based on the consideration of four dimensional intersection forms. This result was later strengthened in \cite[Corollary 1.3]{CNT}, also by using four dimensional techniques.

Our goal is now to apply the strategy from the knot case to the signature defined by Cimasoni and Florens. More precisely, we observe that the definition of Cimasoni and Florens makes sense for any family of matrices in $\mathcal{M}_\mu^\ast$, and that, by their computations, it is invariant by generalized S-equivalence. We then use the algebraic fact that that the signature of a non degenerate hermitian matrix is invariant under Witt equivalence to show that the signature is an algebraic concordance invariant, except at zeros of the Alexander polynomial. In particular, we recover a variation of \cite[Theorem 7.1]{C-F08} and \cite[Corollary 1.3]{CFT18} as Corollary \ref{cor:Signature}, in the smooth category.
\medskip

We start with a definition, inspired by \cite{C-F08}. Recall from Definition \ref{def:MAP} the notation
$$A(t_1, \ldots, t_\mu) := \sum_{\varepsilon \in \{\pm\}^\mu}\varepsilon_1 \ldots \varepsilon_\mu t_1^{\frac{1-\varepsilon_1}{2}} \ldots t_\mu^{\frac{1-\varepsilon_\mu}{2}}A^\varepsilon,$$
for any family of matrices $A \in \mathcal{M}_\mu$.

\begin{definition}
For a family of matrices $A$ in $\mathcal{M}_\mu$ and a vector of complex numbers $\omega := (\omega_1, \ldots, \omega_\mu) \in (S^1\setminus\{1\})^\mu$, the signature $\sigma_A(\omega)$ of $A$ at $\omega$ is the signature of the hermitian matrix
$$H_A(\omega) :=\left(\prod_{i=1}^\mu(1-\overline{\omega_i})\right) A(\omega_1, \ldots, \omega_\mu).$$
\end{definition}

\begin{remark}
\label{rk:Signature}
\begin{enumerate}[(i)]
\item In the case where $A$ is a family of generalized Seifert matrices for a colored link $L$, the signature $\sigma_A(\omega)$ corresponds to the signature $\sigma_L(\omega)$ of $L$. Indeed, this is the definition of \cite[Section 2.2]{C-F08}.
\item By definition, the signature is additive, i.e. it satisfies $\sigma_{A\oplus B}(\omega) = \sigma_A(\omega) \oplus \sigma_B(\omega)$ for all families of matrices $A$ and $B$ in $\mathcal{M}_\mu$ and all $\omega$ in $(S^1\setminus 1)^\mu$. 
\end{enumerate}
\end{remark}

We now show that the signature is an algebraic concordance invariant, away from the zeros of the Alexander polynomial.

\begin{proposition}
\label{prop:SignatureAlgConcInvariant}
Suppose that $A$ and $B$ are algebraically concordant families of matrices in $\mathcal{M}_\mu^\ast$. Then for any $\omega \in (S^1\setminus \{1\})^\mu$ such that $\Delta_A(\omega)$ and $\Delta_B(\omega)$ are non zero, we have $\sigma_A(\omega) = \sigma_B(\omega)$.
\end{proposition}

\begin{proof}
Invariance by generalized S-equivalence of the signature was already established in \cite[Theorem 2.1]{C-F08} for the first three transformations, and by \cite[Theorem 3.2]{DMO21} for the fourth transformation. Note that these two results are stated for families of matrices coming from colored links, but the proofs deal with generalized S-equivalence, and thus hold for any family of matrices in $\mathcal{M}_\mu^\ast$.

Thus, to prove the proposition, it only remains to see that if families $A$ and $B$ are Witt equivalent, with $\Delta_A(\omega)$ and $\Delta_B(\omega)$ both non zero, then the signatures $\sigma_A(\omega)$ and $\sigma_B(\omega)$ coincide. 
Note that in this case, the equality $\Delta_A(\omega) \;\ddot=\; \det(H_A(\omega))$ (immediate from the definition) implies that $H_A(\omega)$ is a non degenerate matrix. Similarly, $H_B(\omega)$ is a non degenerate matrix. Therefore, $H_A(\omega)$ and $H_B(\omega)$ are Witt equivalent, non degenerate, hermitian matrices and thus have the same signature.
\end{proof}

Now that it has been established that the signature is an algebraic concordance invariant, the invariance of the signature for links is a direct consequence of Theorem \ref{prop:AlgConcLinks}.

\begin{proof}[Proof of Corollary \ref{cor:Signature}]
Suppose that $L$ and $L'$ are smoothly concordant colored links. First, note that if $\Delta_L$ is trivial (or equivalently if $\Delta_{L'}$ is trivial), then the statement is trivially satisfied. 

It remains to deal with the case where both Alexander polynomials are non trivial. By Theorem \ref{prop:AlgConcLinks}, the links $L$ and $L'$ are algebraically concordant. By definition, this means that any of their families of generalized Seifert matrices $A$ and $A'$ are algebraically concordant. Thus, by Proposition \ref{prop:SignatureAlgConcInvariant}, the signatures $\sigma_A(\omega)$ and $\sigma_{A'}(\omega)$ coincide provided $\omega$ is not a root of $\Delta_A(t)$ or $\Delta_{A'}(t)$. As all the coordinates of $\omega$ are different from one, this is equivalent to the fact that $\omega$ is not a root of the Alexander polynomials of $L$ and $L'$. As the signatures $\sigma_A(\omega)$ and $\sigma_{A'}(\omega)$ are precisely the signatures of $L$ and $L'$ (recall Remark \ref{rk:Signature} (i)), away from the zeros of the Alexander polynomials, we have
$$\sigma_L(\omega) =  \sigma_A(\omega) = \sigma_{A'}(\omega) = \sigma_{L'}(\omega),$$
and the corollary is proved.
\end{proof}

\subsection{The Fox-Milnor condition}
\label{sub:FoxMilnor}
In this section, after reviewing the classical case, we show that the Alexander polynomials of algebraically concordant families of matrices satisfy the Fox-Milnor condition. We then connect this with links by proving Corollary \ref{cor:FoxMilnor}.
\medskip

In the classical case of knots, the so-called \emph{Fox-Milnor condition} states that if two knots $K_1$ and $K_2$ are concordant, then their Alexander polynomial must satisfy 
$$\Delta_{K_1}(t)\Delta_{K_2}(t) \,\dot=\, h(t)h(t^{-1})$$
for some integer Laurent polynomial $h(t)$. This is fairly easy to establish once it has been shown that concordant knots are algebraically concordant. Indeed, the product of the Alexander polynomials of algebraically concordant knots can be computed via a metabolic Seifert matrix, and a direct computation shows that the Alexander polynomial associated to such a matrix is of the form $h(t)h(t^{-1})$. 

The Fox-Milnor condition has also been established for the multivariable Alexander polynomial of concordant links by Kawauchi \cite{Kawauchi} using homological techniques (see also \cite{NakagawaAP} for a proof via Fox calculus). This correspond to the case of ordered links, where the numbers of colors coincides with the number of components. As the Alexander polynomial of a colored link can be recovered by specializing the variables of the Alexander polynomial of the underlying ordered link, Kawauchi's result readily implies the Fox-Milnor condition for colored links.

The aim of the present section is to obtain an analogue of this result for families of matrices in $\mathcal{M}_\mu^\ast$, by generalizing the elementary proof from the classical case of knots. Before we do so, let us recall that for $\mu >1$, the Alexander polynomial of a family of matrices is invariant under generalized S-equivalence only up to multiplication by units of $\Lambda_S$. Therefore, the best we can aim for with our techniques is a Fox-Milnor condition up to multiplication by factors $(t_i-1)$. This is what we get.

\begin{proposition}
\label{prop:FoxMilnor}
Suppose that $A$ and $B$ are algebraically concordant families of matrices in $\mathcal{M}_\mu^\ast$. Then 
$\Delta_A(t_1, \ldots, t_\mu)\Delta_B(t_1, \ldots, t_\mu) \,\ddot =\, h(t_1, \ldots, t_\mu)h(t_1^{-1}, \ldots, t_\mu^{-1})$
for some integral Laurent polynomial $h \in \Lambda$.
\end{proposition}

\begin{proof}
We first show that the Fox-Milnor condition is transitive on the set of symmetric polynomials : if $\Delta_A(t)$, $\Delta_B(t)$ and $\Delta_C(t)$ are symmetric polynomials satisfying $\Delta_A(t) \Delta_B(t) \,\ddot =\, f(t)f(t^{-1})$ and $\Delta_B(t)\Delta_C(t) \,\ddot =\, g(t)g(t^{-1})$, then $\Delta_A(t)\Delta_C(t) \,\ddot =\, h(t)h(t^{-1})$ for some Laurent polynomial $h$. Indeed, under these hypotheses, as $\Delta_B(t)$ is symmetric, we have 
$$\Delta_A(t)\Delta_C(t)\cdot\Delta_B(t)\Delta_B(t^{-1}) \,\ddot =\, \Delta_A(t)\Delta_B(t)\cdot\Delta_B(t)\Delta_C(t) \,\ddot =\,  f(t)g(t)f(t^{-1})g(t^{-1}).$$
We claim that we can divide the right hand side by $\Delta_B(t)\Delta_B(t^{-1})$, and that the result is of the form $h(t)h(t^{-1})$. Indeed, if $p(t)$ is an irreducible polynomial appearing in the decomposition of $\Delta_B(t)$, then it divides either $f(t)g(t)$, or $f(t^{-1})g(t^{-1})$. In the first case, the factor $p(t^{-1})$, which appears in the decomposition of $\Delta_B(t^{-1})$ divides $f(t^{-1})g(t^{-1})$, and in the second case it divides $f(t)g(t)$. As the ring $\Lambda_S$ is a unique factorization domain, one can divide $\Delta_B(t)$ into such irreducible pieces, and after dividing by all of them, one ends up with 
$$\Delta_A(t)\Delta_C(t) \,\ddot =\, \frac{ f(t)g(t)f(t^{-1})g(t^{-1})}{\Delta_B(t)\Delta_B(t^{-1}) } =: h(t)h(t^{-1})$$
for some Laurent polynomial $h$.

Now that the transitivity of the Fox-Milnor condition has been established, the proof reduces to verifying that the Fox-Milnor condition is satisfied by Alexander polynomials of generalized S-equivalent and Witt equivalent families of matrices. We now verify this condition on the two equivalence relations, starting with generalized S-equivalence.

By Lemma \ref{lemma:APInvariance}, we know that if $A$ and $B$ are related by generalized S-equivalence, then their Alexander polynomials coincide up to multiplication by units of $\Lambda_S$. As Alexander polynomials are symmetric, we have 
$$\Delta_A(t)\Delta_B(t) \,\ddot =\, \Delta_A(t)\Delta_A(t) \,\ddot =\, \Delta_A(t)\Delta_A(t^{-1}),$$
as expected.

Finally, if $A$ and $B$ are Witt equivalent families of matrices, then the family of matrices $A\oplus -B$ is metabolic. Therefore, the family $A\oplus -B$ is congruent to a family of matrices of the form
$$
(A\oplus -B)^\varepsilon \sim \begin{pmatrix} 0 & C^\varepsilon \\ D^\varepsilon & E^\varepsilon \end{pmatrix},
$$
with $D^{-\varepsilon} = (C^\varepsilon)^{\T}$. Hence, 
\begin{equation}
\label{eq:APNorm}
(A\oplus -B)(t) \sim \begin{pmatrix} 0 & C(t) \\ D(t) & E(t) \end{pmatrix},
\end{equation}
where $C(t) = \sum_\varepsilon \varepsilon_1 \cdots \varepsilon_\mu t_1^\frac{1-\varepsilon_1}{2}\cdots t_\mu^{\frac{1-\varepsilon_\mu}{2}} C^\varepsilon$ and $D(t) = \sum_\varepsilon\varepsilon_1 \cdots \varepsilon_\mu t_1^\frac{1-\varepsilon_1}{2}\cdots t_\mu^{\frac{1-\varepsilon_\mu}{2}} D^\varepsilon$. As $D^{-\varepsilon} = (C^\varepsilon)^{\T}$, we have
$$D(t) = (-1)^\mu t C(t^{-1})^{\T}$$
and taking the determinant on both sides of $(\ref{eq:APNorm})$ yields
$$\Delta_A(t) \Delta_B(t) \,\dot=\, \det((A\oplus -B)(t)) \,\dot=\, \det(C(t)) \det( C(t^{-1})),$$
which is of the required form.
\end{proof}

We now connect with links, proving Corollary \ref{cor:FoxMilnor}. Note that if we drop the condition on $f$ and $g$, Corollary \ref{cor:FoxMilnor} is an immediate consequence of Propositions \ref{prop:AlgConcLinks} and \ref{prop:FoxMilnor}. However, in order to establish that $\vert f(1, \ldots, 1) \vert = \vert g(1, \ldots, 1)\vert = 1$, we need slightly finer arguments coming from the proof of Theorem \ref{prop:AlgConcLinks}.

\begin{proof}[Proof of Corollary \ref{cor:FoxMilnor}]
Suppose that $L$ and $L'$ are smoothly concordant colored links. Then by \cite[Lemma 1]{NakagawaAP}, there exists a link $L_{1/2}$ obtained from both $L$ and $L'$ by band summing an unlink. Arguing as in the proof of Theorem \ref{prop:AlgConcLinks}, given a family of generalized Seifert matrices $A$ for $L$, we build a family of generalized Seifert matrices $A_{1/2}$ for $L_{1/2}$. By Lemma \ref{lemma:A1/2APNon0}, the family $A_{1/2}$ satisfies 
$$\Delta_{A_{1/2}}(t) \,\ddot=\, \Delta_A(t)f(t)f(t^{-1})$$
for some Laurent polynomial $f \in \Lambda$ satisfying $\vert f(1, \ldots, 1) \vert = 1$. Similarly, turning the concordance upside down and starting with a family of generalized Seifert matrices $A'$ for $L'$, we build a family of generalized Seifert matrices $A'_{1/2}$ for $L_{1/2}$. Lemma \ref{lemma:A1/2APNon0} applies as well yielding
$$\Delta_{A'_{1/2}}(t) \,\ddot=\, \Delta_{A'}(t)g(t)g(t^{-1})$$
for some Laurent polynomial $g \in \Lambda$ satisfying $\vert g(1, \ldots, 1) \vert = 1$. 
As the families $A_{1/2}$ and $A'_{1/2}$ both are families of generalized Seifert matrices for $L_{1/2}$, their Alexander polynomials coincide up to multiplication by units of $\Lambda_S$. Thus we have
$$\Delta_A(t)f(t)f(t^{-1}) \,\ddot=\, \Delta_{A_{1/2}}(t) \,\ddot=\, \Delta_{A'_{1/2}}(t) \,\ddot=\, \Delta_{A'}(t)g(t)g(t^{-1}).$$
The corollary now follows as $A$ and $A'$ are families of generalized Seifert matrices for $L$ and $L'$ respectively.
\end{proof}

\begin{remark}
At first sight, the results of Proposition \ref{prop:FoxMilnor} and Corollary \ref{cor:FoxMilnor} might appear a little different. Indeed, the first one establishes an equality of the form
$$\Delta_A(t)\Delta_B(t) \,\ddot=\, h(t)h(t^{-1}),$$
while the second one gives an equality
\begin{equation}
\label{eq:FoxMilnor}
\Delta_L(t) f(t)f(t^{-1}) \,\ddot=\, \Delta_{L'}(t) g(t)g(t^{-1}).
\end{equation}
It is easy to see that, as Alexander polynomials are symmetric, these two formulations are in fact equivalent. However, the main difference between Proposition \ref{prop:FoxMilnor} and Corollary \ref{cor:FoxMilnor} resides in the fact that the second one deals with concordance of colored links, and contains a condition on the polynomials $f$ and $g$, while the first one holds in the more general context of algebraic concordance of families of matrices, and does not involve any condition on the polynomial $h$. 

In fact, an equality such as (\ref{eq:FoxMilnor}) with $\vert f(1, \ldots, 1) \vert = \vert g(1, \ldots, 1) \vert = 1$ does not hold in general for algebraically concordant families of $\mathcal{M}_\mu^\ast$. For instance, the matrix $A^+ = \bsm 0&2\\0&0 \esm$ determines a metabolic family $A=\{A^+, (A^+)^{\T}\}$ of $\mathcal{M}_1^\ast$. Thus, $A$ is algebraically concordant to the family $H$ of empty matrix. A direct computation shows that $\Delta_A(1) = 4$ and $\Delta_H(1) = 1$. Thus, the Alexander polynomials of the algebraically concordant families $A$ and $H$ are not related as in (\ref{eq:FoxMilnor}) with polynomials $f$ and $g$ satisfying $\vert f(1)\vert = \vert g(1) \vert = 1$.
\end{remark}

\subsection{The Blanchfield pairing}
\label{sub:Blanchfield}
In this section, after briefly recalling the classical Blanchfield pairing of a knot and its extension to colored links, we define the Blanchfield pairing of families of matrices in $\mathcal{M}_\mu^\ast$ and show that its Witt equivalence class is an algebraic concordance invariant. We then apply this result to links, proving Corollary \ref{cor:Blanchfield}.
\medskip

The Blanchfield pairing of a knot, introduced by Blanchfield in \cite{Bla57}, is a non singular hermitian pairing on the Alexander module of the knot. Even though it was originally defined homologically, it can be explicitely computed (an therefore defined) as follows. Given a Seifert matrix $A$ for a knot, the Blanchfield pairing is isometric to the pairing 
$$\Lambda^n / H^{\T} \Lambda^n \times \Lambda^n / H^{\T} \Lambda^n \to Q(t)/\Lambda, \quad ([x], [y]) \mapsto [-x^{\T}H^{-1}\overline{y}],$$
where $H = (1-t)A^{\T} + (1-t^{-1})A$. Because of its homological definition, the isometry class of the Blanchfield pairing is a knot invariant. Moreover, Kearton \cite{Kea75} showed that a knot is algebraically slice if and only if its Blanchfield pairing is \emph{metabolic}, a notion that will be made precise in this section. Note that the \emph{only if} part follows immediately from the fact that the Blanchfield pairing can be represented by a metabolic matrix.

The Blanchfield pairing extendeds to colored links, see e.g. \cite[Section 2.4]{hillman2012}. In \cite{CFT18}, Theorem 1.2 provides an explicit formula for the computation of the Blanchfield pairing using totally connected C-complexes. The totally connected hypothesis was later showed to be unnecessary by Conway, see \cite[Theorem 1.1]{ConwayBlanchfield}. The Witt equivalence class of the Blanchfield pairing of ordered links is a concordance invariant, see \cite[Theorem 2.4]{hillman2012}. The proof is homological.

The goal of this section is to use the explicit formula from \cite{CFT18} to define a Blanchfield pairing for all families of matrices in $\mathcal{M}_\mu^\ast$. We then show that the Witt class of this pairing is an algebraic concordance invariant.
\medskip

We start with a definition, inspired by \cite{CFT18}. In order to simplify the notation, we write $H_A:= H_A(t)$, and denote by $x \mapsto \overline{x}$ the involution of $\Lambda_S$ inverting the powers of all $t_i$. We write $Q$ for the field of fraction of $\Lambda_S$.

\begin{definition}
\label{def:Blanchfield}
The \emph{Blanchfield pairing} of a family of matrices $A \in \mathcal{M}_\mu^\ast$ is the non degenerate hermitian pairing
$$\Bl_A \colon \Lambda^n_S / H_A^{\T} \Lambda_S^n \times \Lambda^n_S / H_A^{\T} \Lambda^n_S \to Q/\Lambda_S, \quad ([x], [y]) \mapsto [-x^{\T}H_A^{-1}\overline{y}].$$
\end{definition}

\begin{remark}
\label{rk:Blanchfield}
\begin{enumerate}[(i)]
\item The notation $H_A^{-1}$ makes sense as the Alexander polynomial $\Delta_A(t) \,\ddot=\, \det(H_A)$ is non trivial.
\item A direct computation shows that the Blanchfield pairing is well defined and indeed non degenerate. The presence of the word hermitian in the definition is clear as $H_A$ itself is hermitian.
\item In the case where $A$ is a family of generalized Seifert matrices for a colored link $L$, the Blanchfield pairing $\Bl_A$ coincides with the Blanchfield pairing of $L$ over $\Lambda_S$, by \cite[Theorem 1.2]{CFT18}
\end{enumerate}
\end{remark}

As already stated, the Blanchfield pairing itself is not a concordance invariant, but its Witt equivalence class is, a notion that we now briefly recall. Details can be found in \cite[Chapters 2.3 and 2.8]{hillman2012} or in \cite{GSThesis}. A \emph{linking pairing} over $\Lambda_S$ is a hermitian pairing $B \colon M \times M \to Q/\Lambda_S$ for a certain torsion $\Lambda_S$-module $M$. A non degenerate linking pairing is called $\emph{metabolic}$ if there is a submodule $N$ of $M$ such that $N = N^{\perp} := \{ x \in M \vert B(x,y) = 0 \text{ for all } y \in N\}$. Two non degenerate linking pairings $B$ and $B'$ are called \emph{Witt equivalent} if there exists metabolic pairings $C$ and $C'$ such that $B \oplus C$ and $B' \oplus C'$ are isometric. It is well known that Witt equivalence is an equivalence relation on the set of non degenerate linking pairings, and that Witt equivalence classes of non degenerate linking pairings form an abelian group for the operation induced by direct sum (see for instance \cite[Theorem 2.3]{hillman2012} for a formulation in the non singular case, or \cite{GSThesis} for non degeneracy). The neutral element corresponds to the class of metabolic pairings, and the opposite class of a pairing $B$ is given by the class of the opposite pairing $-B$. We denote this group by $W(Q/\Lambda_S)$.

Using this notion, we have the following proposition.

\begin{proposition}
\label{prop:InvarianceBlanchfield}
The map $\Bl \colon \mathcal{C}_\mu^\alg \to W(Q/\Lambda_S)$, which associates to the class of a family of matrices the class of its Blanchfield pairing, is a well defined group homomorphism.
\end{proposition}

\begin{remark}
By remark \ref{rk:AlgConcGroupBig} (ii), for $\mu = 1$, the group $\mathcal{C}_1^\alg$ contains the classical algebraic concordance group of knots $\mathcal{C}^\alg_\mathit{knot}$. The map of Proposition \ref{prop:InvarianceBlanchfield} is compatible with the classical Blanchfield pairing of knots in the sense that the following diagram commutes : 
$$\begin{tikzcd}
\mathcal{C}^\alg_\mathit{knot} \arrow{r} \arrow{d}{\Bl_\mathit{knot}} & \mathcal{C}_1^\alg \arrow{d}{\Bl}\\
W_{\text{n.s.}}(Q/\Lambda) \arrow{r} & W(Q/\Lambda_S).
\end{tikzcd}$$
The top horizontal map is induced by the inclusion, see Remark \ref{rk:AlgConcGroupBig} (ii). The first vertical map is given by the classical Blanchfield pairings of knots \cite{Bla57}, while the second one is the map from Proposition \ref{prop:InvarianceBlanchfield}. The bottom left group denotes the Witt group of non singular linking forms over $\Lambda$, and the bottom horizontal map is induced by the inclusion of $\Lambda$ into $\Lambda_S$. The fact that the diagram commutes follows from the fact that Definition \ref{def:Blanchfield} extends a well known formula for the Blanchfield pairing of knots, see \cite{CFT18}.
\end{remark}

\begin{proof}[Proof of Proposition \ref{prop:InvarianceBlanchfield}]
To prove that this map is well defined, we show that generalized S-equivalent and Witt equivalent families of matrices have Witt equivalent Blanchfield pairings. 
We start with generalized S-equivalence. Suppose that $A$ and $A' \in \mathcal{M}_\mu^\ast$ are related by one of the transformations (T1) through (T4). By \cite[Proof of Theorem 2.1]{C-F08} (see also \cite[p.14]{ConwayBlanchfield}), if $A'$ is obtained from $A$ by one of the transformations (T1) or (T2), then $H_{A'}$ and $H_A$ are related by \emph{elementary enlargements}, i.e. transformations of the form
$$H \mapsto \begin{pmatrix} H & \xi & 0 \\ \overline{\xi}^{\T} & \lambda & \alpha \\ 0 & \overline{\alpha} & 0\end{pmatrix},$$
with $\alpha$ a unit of $\Lambda_S$. Up to a change of basis, elementary enlargements correspond to the transformation $H \mapsto H \oplus \bsm0&1\\1&0\esm$. A direct computation shows that this transformation does not change the isometry type of the Blanchfield form. Thus, the isometry class of the Blanchfield form is unchanged by the transformations (T1) and (T2).

Following \cite{C-F08}, the third transformation can be dealt with similarly. For the sake of completeness, we give details. By definition, if $A$ and $A' \in \mathcal{M}_\mu^\ast$ are related by a transformation (T3), then the matrices $A(t)$ and $A'(t)$ are related via
$$A(t) = \begin{pmatrix}
A_0 & -\sigma_it_i^\frac{1+\sigma_i}{2} \prod_{l\neq i}(1-t_l)\chi(z) & \xi\\
\sigma_it_i^\frac{1-\sigma_i}{2} \prod_{l\neq i} (1-t_l)\chi(z)^{\T} & 0 & \sigma_i\sigma_jt_i^\frac{1-\sigma_i}{2}t_j^\frac{1-\sigma_j}{2} \prod_{l\neq i,j} (1-t_l)\\
\overline{\xi}^{\T} & \sigma_i\sigma_jt_i^\frac{1+\sigma_i}{2} t_j^\frac{1+\sigma_j}{2} \prod_{l\neq i,j}(1-t_l) & \ell
\end{pmatrix}$$
and $A'(t)$ the matrix
$$\footnotesize\begin{pmatrix}
A_0 & -\sigma_kt_k^\frac{1+\sigma_k}{2} \prod_{l\neq k}(1-t_l)\chi(z) & \xi + \sigma_i\sigma_kt_i^\frac{1-\sigma_i}{2}t_k^\frac{1+\sigma_k}{2}\prod_{l\neq i,k}(1-t_l) \chi(z)\\
\sigma_kt_k^\frac{1-\sigma_k}{2} \prod_{l\neq k} (1-t_l)\chi(z)^{\T} & 0 & \sigma_j\sigma_kt_j^\frac{1-\sigma_j}{2}t_k^\frac{1-\sigma_k}{2} \prod_{l\neq j,k} (1-t_l)\\
\overline{\xi}^{\T} +\sigma_i\sigma_kt_i^\frac{1+\sigma_i}{2}t_k^\frac{1-\sigma_k}{2}\prod_{l\neq i,k}(1-t_l) \chi(z)^{\T} & \sigma_j\sigma_kt_j^\frac{1+\sigma_j}{2} t_k^\frac{1+\sigma_k}{2} \prod_{l\neq j,k}(1-t_l) & \ell'
\end{pmatrix},$$
where 
$$\ell' = \ell + \sigma_i\sigma_j\sigma_k\left(t_i^\frac{1+\sigma_i}{2} t_j^\frac{1-\sigma_j}{2} t_k^\frac{1-\sigma_k}{2} -  t_i^\frac{1-\sigma_i}{2} t_j^\frac{1+\sigma_j}{2} t_k^\frac{1+\sigma_k}{2}\right)\prod_{l\neq i,j,k}(1-t_l).$$
A direct computation shows that the matrix $A'(t)$ is congruent over $\Lambda_S$ to the matrix
$$\begin{pmatrix}
A_0 & -\sigma_kt_k^\frac{1+\sigma_k}{2} \prod_{l\neq k}(1-t_l)\chi(z) & \xi\\
\sigma_kt_k^\frac{1-\sigma_k}{2} \prod_{l\neq k} (1-t_l)\chi(z)^{\T} & 0 & \sigma_j\sigma_kt_j^\frac{1-\sigma_j}{2}t_k^\frac{1-\sigma_k}{2} \prod_{l\neq j,k} (1-t_l)\\
\overline{\xi}^{\T} & \sigma_j\sigma_kt_j^\frac{1+\sigma_j}{2} t_k^\frac{1+\sigma_k}{2} \prod_{l\neq j,k}(1-t_l) & \ell
\end{pmatrix},$$
which itself is clearly congruent to $A(t)$ over $\Lambda_S$. Therefore, the matrices $A(t)$ and $A'(t)$ are congruent over $\Lambda_S$, and it follows that the matrices $H_A$ and $H_{A'}$ are also congruent over $\Lambda_S$. Thus, the induced Blanchfield forms are isometric. 

The last movement (T4) is dealt with similarly. Again, we give details for completeness. Suppose that $A$ and $A'$ are related by a transformation (T4). Then the matrices $A(t)$ and $A'(t)$ are related by 
$$A(t) = \begin{pmatrix}
A_0 & - \sigma_it_i^\frac{1+\sigma_i}{2}\prod_{l\neq i}(1-t_l) \chi(z) & \xi\\
\sigma_i t_i^\frac{1-\sigma_i}{2} \prod_{l\neq i}(1-t_l) \chi(z)^{\T} & 0 & \sigma_i\sigma_jt_i^\frac{1-\sigma_i}{2}t_j^\frac{1-\sigma_j}{2}\prod_{l\neq i,j}(1-t_l) \\
\overline{\xi}^{\T} & \sigma_i\sigma_j t_i^\frac{1+\sigma_i}{2} t_j^\frac{1+\sigma_j}{2}\prod_{l\neq i,j} (1-t_l) & \ell
\end{pmatrix}$$
and
$$A'(t) = \begin{pmatrix}
A_0 & \mp\sigma_i t_i^\frac{1\pm\sigma_i}{2} \prod_{l \neq i} (1-t_l) \chi(z') & \xi + n \prod_l(1-t_l) \chi(z) \\
\pm \sigma_i t_i^\frac{1\mp\sigma_i}{2} \prod_{l\neq i} (1-t_l) \chi(z)^{\T} & 0 & \pm\sigma_i \sigma_j t_i^\frac{1\mp\sigma_i}{2} t_j^\frac{1-\sigma_j}{2} \prod_{l\neq i,j} (1-t_l)\\
\overline{\xi}^{\T} + n \prod_l (1-t_l) \chi(z')^{\T} & \pm \sigma_i\sigma_j t_i^\frac{1\pm \sigma_i}{2} t_j^\frac{1+\sigma_j}{2} \prod_{l\neq i,j} (1-t_l) & \ell + n\prod_l(1-t_l)
\end{pmatrix}.$$
It is not difficult to see that these two matrices are congruent over $\Lambda_S$. Thus, the matrices $H_A$ and $H_{A'}$ are congruent over $\Lambda_S$ and it follows that the corresponding Blanchfield forms are isometric.

Therefore, the isometry class of the Blanchfield form is unchanged by the transformations (T1) through (T4) and thus by generalized S-equivalence. As isometric linking pairings are Witt equivalent (a fact which is immediate from the definition), the Witt equivalence class of the Blanchfield form is invariant by generalized S-equivalence.

We now turn to Witt equivalence of families of matrices. Suppose that two families of matrices $A$ and $A'$ in $\mathcal{M}_\mu^\ast$ are Witt equivalent. Then the direct sum $H_A \oplus -H_{A'}$ is metabolic. Thus, the matrix $(H_A \oplus -H_{A'})^{-1} = H_A^{-1} \oplus -H_{A'}^{-1}$, which represents the pairing $\Bl_A \oplus \Bl_{A'}$ is also metabolic. As the pairing $\Bl_A \oplus \Bl_{A'}$ is represented by a non degenerate metabolic matrix, it is metabolic (see for instance \cite[Appendix C]{FriedlThesis}. Hence, the standard isometry
$$\Bl_A \oplus (\Bl_{A'} \oplus -\Bl_{A'}) \cong \Bl_{A'} \oplus (\Bl_A \oplus -\Bl_{A'})$$
implies that $\Bl_A$ and $\Bl_{A'}$ are Witt equivalent linking forms. 

The fact that $\Bl$ is a homomorphism is immediate as $H_{A \oplus B} = H_{A} \oplus H_{B}$ gives a matrix representing both pairings $\Bl_{A\oplus B}$ and $\Bl_A \oplus \Bl_B$.
\end{proof} 

The fact that smoothly concordant colored links have Witt equivalent Blanchfield pairings now immediately follows from Propositions \ref{prop:AlgConcLinks} and \ref{prop:InvarianceBlanchfield}.

\begin{proof}[Proof of Corollary \ref{cor:Blanchfield}]
Suppose that $L$ and $L'$ are smoothly concordant colored links. Then by Theorem \ref{prop:AlgConcLinks}, they are algebraically concordant. Thus, any of their families of generalized Seifert matrices $A$ and $A'$ are algebraically concordant. By Proposition \ref{prop:InvarianceBlanchfield}, the corresponding Blanchfield forms $\Bl_A$ and $\Bl_{A'}$ are Witt equivalent. By Remark \ref{rk:Blanchfield} (ii), these coincide with the Blanchfield pairings of $L$ and $L'$ over $\Lambda_S$.
\end{proof}

\subsection{Structure of $\mathcal{C}_\mu^\alg$}
\label{sub:Structure}
In this section, we use the signature and the Fox-Milnor condition to prove Theorem \ref{prop:StructureAlgConcGroup}.

Our first step toward the proof of Theorem \ref{prop:StructureAlgConcGroup} is to use the signature to show that $\mathcal{C}_\mu^{\alg}$ contains a summand isomorphic to $\Z^\infty$ for all $\mu \geq 1$. To do so, we copy the classical strategy from the case of knots, which consists of finding an infinite family of knots $(K_i)_{i\geq 1}$, with linearly independent signatures. 

From such a family of knots, it is easy to produce a family of colored link with the same property by attaching an unlink. More precisely, given any $1$-colored link $K$, and any number of colors $\mu \geq 1$, we consider a colored link $K_\mu$ as depicted in the left of Figure \ref{fig:FakeLink}. Note that the isotopy class of $K_\mu$ depends on the choice of a component of $K$. However, for any choice of component, this link has the following useful property.

\begin{lemma}
\label{lemma:FakeLink}
Given any Seifert matrix $A$ for a one colored link $K$, and any integer $\mu \geq 1$, the $\mu$-colored link $K_\mu$ admits the family $A_\mu = \{A, \ldots, A, A^{\T}, \ldots, A^{\T}\}$ as a family of generalized Seifert matrices. In particular, the signatures $\sigma_{A_\mu}(\omega_1, \ldots, \omega_\mu)$ and $\sigma_K(\omega_1)$ coincide and the Alexander polynomials $\Delta_{A_\mu}(t_1, \ldots, t_\mu)$ and $\Delta_K(t_1)$ coincide up to multiplication by units of $\Lambda_S$.
\end{lemma}

\begin{figure}[h]
\centering
\begin{overpic}[width=10cm]{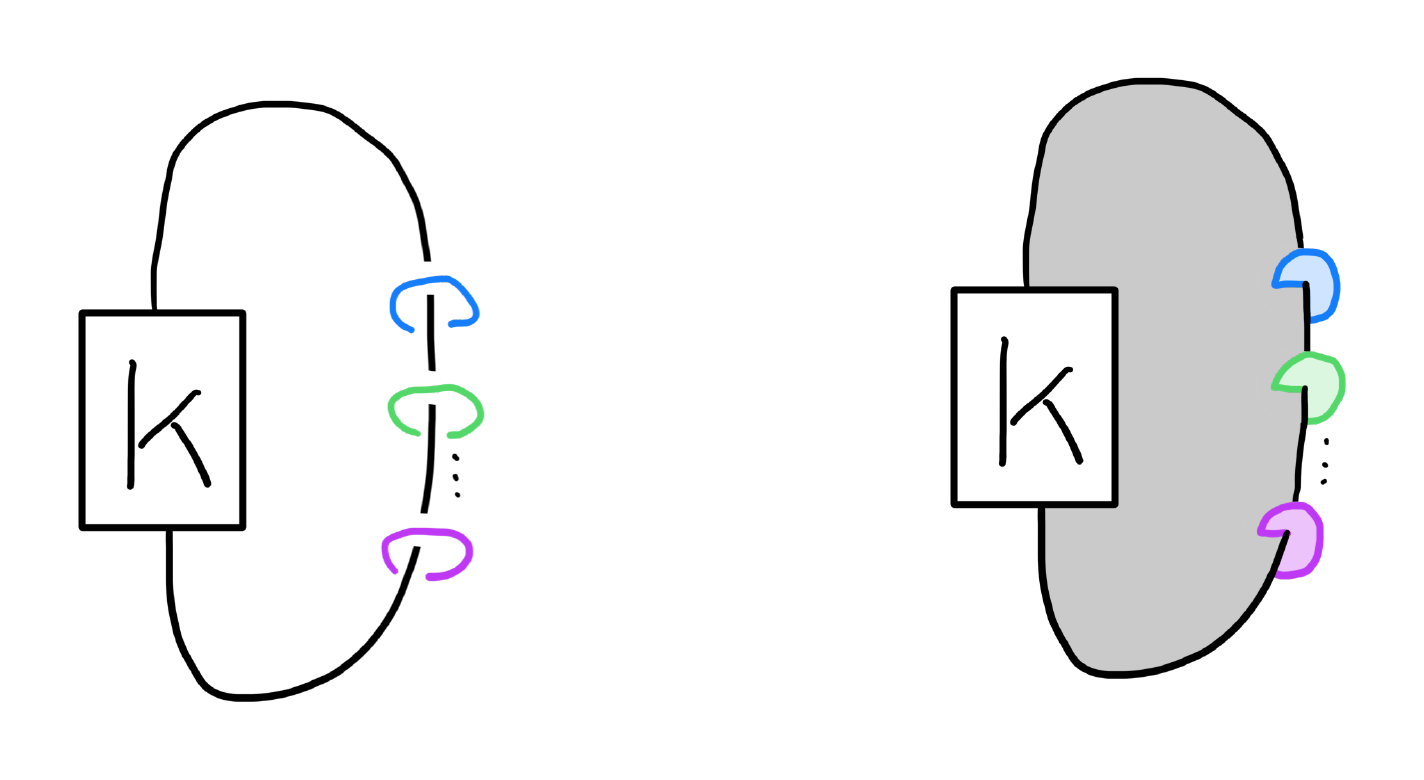}\end{overpic}
\caption{Left, the link $K_\mu$ obtained from $K$ by adding $\mu -1$ unlinked components. Right, a C-complex for $K_\mu$ induced by a Seifert surface for $K$.}
\label{fig:FakeLink}
\end{figure}

\begin{proof}
Observe that any Seifert surface for $K$ naturally induces a C-complex for $K_\mu$, shown in the right of Figure \ref{fig:FakeLink}. It follows that for any Seifert matrix $A$ for $K$, the family 
$$A_\mu:=\{A, \ldots, A, A^{\T}, \ldots, A^{\T}\}$$
is a family of generalized Seifert matrices for $K_\mu$. Moreover, a direct computation shows that 
$$H_{A_\mu}(\omega_1, \ldots, \omega_\mu) =  \left((1-\overline{\omega_1})A + (1-\omega_1)A^{\T}\right) \cdot \sum_{\varepsilon \in \{\pm\}^{\mu-1}} \prod_{i=2}^\mu (1-\overline{\omega_i}^{\varepsilon_i}).$$
The sum on the right equals
$$ \sum_{\varepsilon \in \{\pm\}^{\mu-1}} \prod_{i=2}^\mu (1-\overline{\omega_i}^{\varepsilon_i}) = \prod_{i=1}^\mu 2\Re(1-\omega_i),$$
which is a positive real number as $\omega_i \in S^1\setminus \{1\}$. It follows that the signature of $A_\mu$ at $(\omega_1, \ldots, \omega_\mu)$ coincides with the signature of the matrix $(1-\overline{\omega_1})A + (1-\omega_1) A^{\T}$, which by definition is the signature of $K$ at $\omega_1$. Thus, we have the equality
$$\sigma_{A_\mu}(\omega_1, \ldots, \omega_\mu) = \sigma_K(\omega_1),$$
for any $(\omega_1, \ldots, \omega_\mu) \in (S^1\setminus \{1\})^\mu$.

Similarly, we get 
$$\Delta_{A_\mu}(t_1, \ldots, t_\mu) = \det \left( (A- t_1 A^{\T})\prod_{i=2}^\mu(t_i-1)\right) =\Delta_K(t_1) \prod_{i=2}^\mu (t_i-1)^n.$$
Thus, the Alexander polynomial of $A$ and $K$ coincide up to multiplication by factors $t_i-1$. In particular, they are equal up to multiplication by units of $\Lambda_S$.
\end{proof}

In order to produce a family of colored links with independent signatures, we will apply the construction of Lemma \ref{lemma:FakeLink} to a family of knots with independent signatures. To introduce such a family, we consider the following examples, which are taken from \cite{LivingstonKnotConcordance}. 

\begin{example}
\label{example:K(a,b,c)}
Consider the black knot $K(a,b,c)$ shown in Figure \ref{fig:K(a,b,c)}. For the obvious Seifert surface, it admits the matrix 
$$A = \begin{pmatrix} a & (c+1)/2 \\ (c-1)/2 & b \end{pmatrix}$$
as a Seifert matrix. If the gray component is added, one obtain the $1$-colored link $\tilde K(a,b,c)$, which admits the matrix $A \oplus (-1)$ as a Seifert matrix.

\begin{figure}[h!]
\centering
\begin{overpic}[width=12cm]{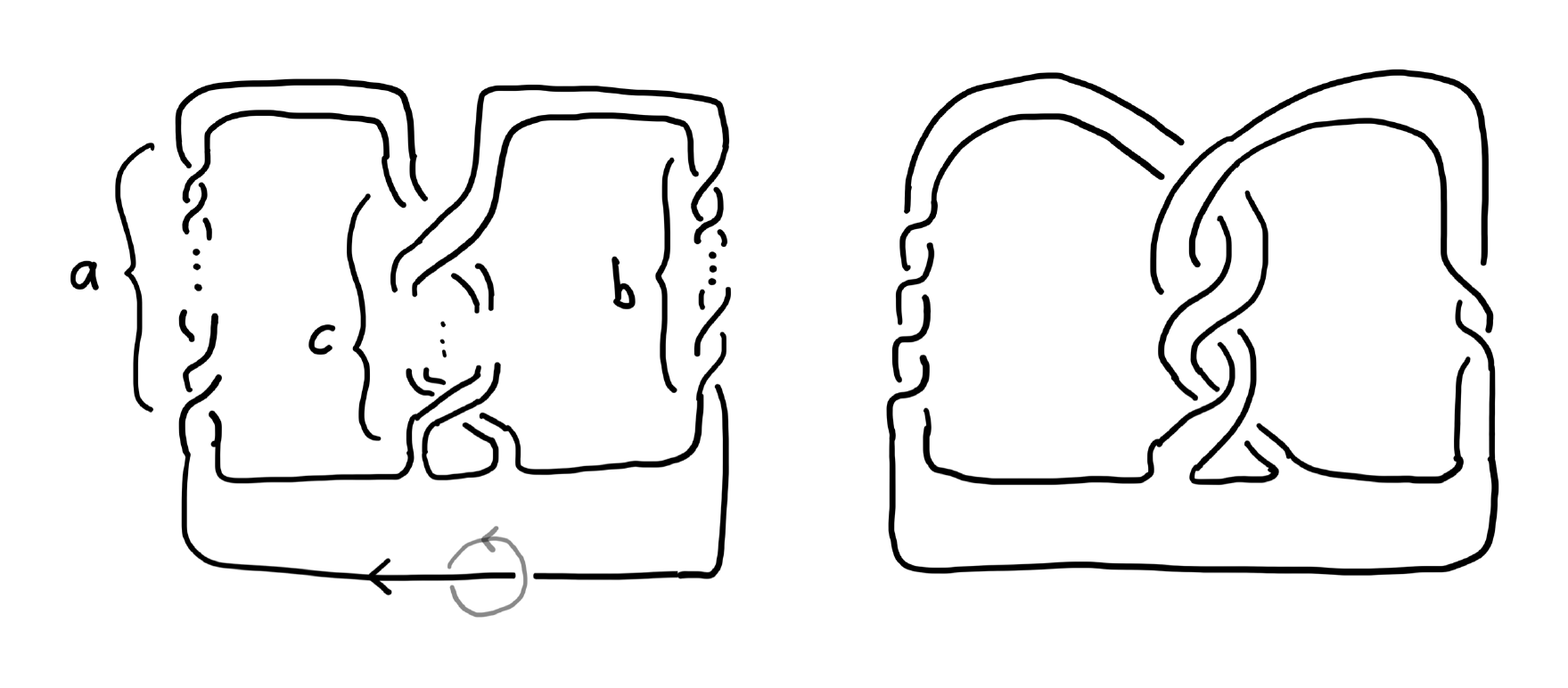}\end{overpic}
\caption{In black, the knot $K(a,b,c)$, where the integers $a$ and $b$ indicate the number of full twists in each band.
The integer $c$ is odd and represents the number of half twists between the bands. The link $\tilde K(a, b, c)$ is obtained by taking the union of $K(a, b, c)$ with the gray circle. Right, the knot $K(2,-1,3)$ as an example.}
\label{fig:K(a,b,c)}
\end{figure}

\end{example}

Using these constructions, we prove the existence of a $\Z^\infty$ summand in $\mathcal{C}_\mu^{\alg}$. 

\begin{lemma}
\label{lemma:ZinfSummand}
For all $\mu \geq 1$, the group $\mathcal{C}_\mu^\alg$ contains a summand isomorphic to $\Z^\infty$
\end{lemma}

\begin{proof}
Let us write $K_n = \tilde K(1,n,1)$. A direct computation using the matrices or Example \ref{example:K(a,b,c)} shows that the signatures of the $1$-colored links $K_n$ are linearly independent. More precisely, let $T = S^1\setminus(\{1\}\cup\{\omega_1 : \Delta_{K_n}(\omega_1) = 0 \text{ for some } n\geq 1\})$ be the set of modulus one complex numbers different from one, and where all roots of all Alexander polynomials of the $K_n$ have been removed. Note that this set is dense in $S^1$. Then the maps
$$\sigma_{K_n} \colon T \to \Z,$$
are linearly independent elements of the $\Z$-module of maps from $T$ to $\Z$.

By Lemma \ref{lemma:FakeLink}, the $\mu$-colored link $(K_n)_\mu$ has Alexander polynomial with the same roots as $\Delta_{K_n}(t_1)$ in $(S^1\setminus\{1\})^\mu$. Thus, the set $T \times (S^1\setminus \{1\})^{\mu-1}$ contains no root of any of the Alexander polynomials of the $(K_n)_\mu$. By Proposition \ref{prop:SignatureAlgConcInvariant}, we get maps
$$\sigma_{(K_n)_\mu} = \sigma_{A_n} \colon T \times (S^1\setminus \{1\})^{\mu-1} \to \Z,$$
which only depend on the algebraic concordance class of a family of generalized Seifert matrices $A_n$ for $(K_n)_\mu$. Moreover, these maps are linearly independent. Indeed, by Lemma \ref{lemma:FakeLink}, their restriction to $T$ coincide with the maps $\sigma_{K_n}$, which are the linearly independent signatures of the links $K_n$. Furthermore, as the signature is additive (see Remark \ref{rk:Signature} (ii)), the linear independence of the maps $\sigma_{A_n}$ implies that the classes $[A_n]$ are linearly independent over $\Z$. Thus, they generate a subgroup of $\mathcal{C}_\mu^{\alg}$ isomorphic to~$\Z^\infty$. 

It remains to show this is a summand. For this, a direct computation using the matrix from Example \ref{example:K(a,b,c)} shows that for all $n$, there exists an $\omega \in (S^1\setminus \{1\})^{\mu}$ such that $\sigma_{A_n}(\omega) = 1$ and $\Delta_{A_n}(\omega) \neq 0$ (this is our motivation for the consideration of $\tilde K(1,n,1)$ rather than $K(1,n,1)$). Thus, if $[A_n]$ decomposes as 
$$[A_n] = [B] \oplus \cdots \oplus [B] = m [B]$$
for a certain family of matrices $B$, we have
$$1 = \sigma_{A_n}(\omega) = m \sigma_B(\omega),$$
and it follows that $m = 1$.
Therefore $[A_n]$ generates a $\Z$-summand of $\mathcal{C}_\mu^{\alg}$ for all $n$. Hence, the subgroup generated by all classes $[A_n]$ is a $\Z^\infty$-summand of $\mathcal{C}_\mu^{\alg}$.
\end{proof}

Now that the existence of a summand of $\mathcal{C}_\mu^\alg$ has been established, we move to the second step of the proof of Proposition \ref{prop:StructureAlgConcGroup}, which consists in showing the existence of a $\Z_2^\infty$ summand via the Fox-Milnor condition. To do so, we copy the classical strategy from the knot case which is to find an infinite family of non algebraically slice, amphichiral knots, which are distinguished by the Fox-Milnor condition. Via a construction similar as for the signature, we get the following result.

\begin{lemma}
\label{lemma:Z2Summand}
For all $\mu\geq 1$, the group $\mathcal{C}_\mu^{\alg}$ admits a summand isomorphic to $\Z_2^\infty$.
\end{lemma}

\begin{proof}
From the family of knots $K_n := K(n, -n, 1)$ for $n\geq 0$, we build a family of colored links $(K_n)_\mu$ as in Figure \ref{fig:FakeLink}. By Example \ref{example:K(a,b,c)}, the knot $K_n$ admits the Seifert matrix
$$A = \begin{pmatrix} n & 1 \\ 0 & -n \end{pmatrix}.$$
Thus, by Lemma \ref{lemma:FakeLink} the link $(K_n)_\mu$ admits the family of generalized Seifert matrices $A_n := \{A, \ldots, A, A^{\T}, \ldots, A^{\T}\}$. As $K_n$ is amphichiral, $A$ is of order at most two in the classical algebraic concordance group and thus $A_n$ is of order at most two in $\mathcal{C}_\mu^{\alg}$. Moreover, Lemma \ref{lemma:FakeLink} and a direct computation using the matrix $A$ show that 
$$\Delta_{A_n}(t_1, \ldots, t_\mu) \,\ddot=\, \Delta_{K_n}(t_1) = n^2t_1^2-(2n^2+1)t_1+n^2.$$
It is easy to check that for all $n \geq 1$ this degree two polynomial is irreducible in $\Lambda_S$. Thus, by Proposition \ref{prop:FoxMilnor}, the class $[A_n]$ is non trivial in $\mathcal{C}_\mu^\alg$. Therefore, each class $[A_n]$  generates a subgroup of $\mathcal{C}_\mu^\alg$ isomorphic to $\Z_2$.

We now use the Fox-Milnor condition to show that these $\Z_2$-subgroups in fact form a $\Z_2^\infty$ summand of $\mathcal{C}_\mu^\alg$. First, as the Alexander polynomial $\Delta_{A_n}(t)$ is irreducible, the class $[A_n]$ cannot be expressed as $m [B]$ for some integer $m >1$ and $B \in \mathcal{M}_\mu^\ast$. It follows that each $\Z_2$-subgroup of $\mathcal{C}_\mu^\alg$ generated by a class $[A_n]$ is a summand. Moreover, as these Alexander polynomials are all irreducible and distinct, the multiplicativity of the Alexander polynomial (see Remark \ref{rk:APnon0Knots} (i)) implies that the classes $[A_n]$ are linearly independent over $\Z_2$. Therefore, they generate a summand of $\mathcal{C}_\mu^\alg$ isomorphic to $\Z_2^{\infty}$. 
\end{proof}

\begin{proof}[Proof of Theorem \ref{prop:StructureAlgConcGroup}]
This is an immediate consequence of Lemma \ref{lemma:ZinfSummand} and Lemma \ref{lemma:Z2Summand}.
\end{proof}

\subsection{Invariants from Witt groups of finite fields}
\label{sub:WittFinite}
In the classical case of knots, Levine \cite{LevInvKnotCob} showed that the algebraic concordance group contains a further $\Z_4^\infty$ summand, by considering invariants coming from Witt group of finite fields. In this section, we give a reason why the extension of these invariants to $\mathcal{C}_\mu^\alg$ does not seem straightforward. This is an obstacle to showing the existence of a $\Z^\infty_4$ summand of $\mathcal{C}_\mu^\alg$.

In the classical case of knots, these invariants are defined as follows. Start with a Seifert matrix $A$ for a knot, and consider the non singular symmetric form $A + A^{\T}$. This defines an element in the Witt group $W(\Q)$, and thus defines a homomorphism $h$ from the classical algebraic concordance group into $W(\Q)$. Then for any prime number $p$, there are well defined map $\partial_p \colon W(\Q) \to W(\mathbb{F}_p)$, and the latter Witt group of a finite field is isomorphic either to $\Z_2$ or $\Z_4$ depending on the prime number $p$. The composition $\partial_p \circ h$ then produces a $\Z_2$ or $\Z_4$ invariant of the classical algebraic concordance group.

Extending this construction to $\mathcal{C}_\mu^\alg$ does not seem straightforward, as the sum of generalized Seifert matrices (even considered with some signs) is not necessarily non singular. Hence, one does in general not get an element in $W(\Q)$. Furthermore, even in the case where one would get an element in $W(\Q)$, one would in general not obtain a well defined map $\mathcal{C}_\mu \to W(\Q)$. Indeed, the Witt class might depend on the choice of the representative of the generalized S-equivalence class. This issue does not appear in the case of knots, as in this classical case S-equivalence implies Witt equivalence.

\section{Questions}
\label{sec:Questions}
We conclude this work by gathering a few questions.\medskip

Our first question concerns the presence of the \emph{smooth} hypothesis in Theorem \ref{prop:IntroAlgConcLinks}. This hypothesis might be unsettling for the reader used to Seifert forms and abelian invariants of links. Indeed, Seifert matrices are classically considered in the topological category, yielding topological invariants. For instance, it is known that the signature of colored links is a topological concordance invariant \cite{C-F08, CNT}, as well as the Fox-Milnor condition \cite{Kawauchi} and the Witt class of the Blanchfield pairing, see for example \cite{hillman2012}. Moreover, in the classical case of knots, the analogue of Theorem \ref{prop:IntroAlgConcLinks} holds in the topological category.  Thus, we ask the following question.

\begin{question}
Do topologically concordant links have algebraically concordant generalized Seifert matrices ?
\end{question}

Our second question is about which elements of the groups $\mathcal{C}_\mu^\alg$ come from generalized Seifert matrices. In the classical case, it is well known that an integral matrix $A$ is a Seifert matrix of a knot if and only if $\det(A-A^{\T}) = 1$. It is therefore very natural to ask whether such a criterion exists for families of generalized Seifert matrices. Our attempts have so far been unfruitful, and we are left with the following question.

\begin{question}
Which families of matrices are realizable as generalized Seifert matrices of links with non trivial Alexander polynomial ?
\end{question}

Our last question was already briefly discussed in Section \ref{sub:WittFinite}, which explain its main motivation. We recall it here.

\begin{question}
\label{question:Z4}
Does the group $\mathcal{C}_\mu^\alg$ admit a $\Z_4^\infty$ summand ?
\end{question}

\bibliographystyle{annotate}
\bibliography{BlanchfieldBib}

\end{document}